\documentclass[reqno,12pt]{amsart} 

\usepackage{amssymb,mathrsfs,bbm}
\usepackage[colorlinks=true, pdfstartview=FitV, linkcolor=blue, citecolor=blue, urlcolor=blue,pagebackref=false]{hyperref}

\usepackage{tikz}
\usetikzlibrary{matrix}
\usepackage{float}
\usepackage[font=footnotesize]{caption}

\usepackage{esint}
\usepackage{bookmark}

\usepackage[left=1in, right=1in, bottom=1.5in]{geometry} 
\newtheorem{theorem}{Theorem}[section]
\newtheorem{lemma}[theorem]{Lemma}
\newtheorem{proposition}[theorem]{Proposition}
\newtheorem{corollary}[theorem]{Corollary}
\theoremstyle{definition}
\newtheorem{definition}[theorem]{Definition}

\theoremstyle{remark}
\newtheorem{remark}[theorem]{Remark}

\numberwithin{equation}{section}

\newcommand{\norm}[1]{\lVert#1\rVert}
\newcommand{\Norm}[1]{\left\lVert#1\right\rVert}

\newcommand{\R}{\mathbb{R}}

\newcommand{\e}{\varepsilon}
\newcommand{\txt}[1]{\text{#1}}
\newcommand{\dist}[0]{\txt{dist}}
\begin{document}

\title[Transmission problem and homogenization]{Regularity of a transmission problem and periodic homogenization}


\author{Jinping Zhuge}
\address{Department of Mathematics, University of Chicago, Chicago, IL, 60637, USA.}
\curraddr{}
\email{jpzhuge@math.uchicago.edu}



\begin{abstract}
This paper is concerned with the regularity theory of a transmission problem arising in composite materials. We give a new self-contained proof for the $C^{k,\alpha}$ estimates on both sides of the interface under the minimal assumptions on the interface and data. Moreover, we prove the uniform Lipschitz estimate across a $C^{1,\alpha}$ interface when the coefficients on both sides of the interface are periodic with independent structures and oscillating at different microscopic scales. 
\end{abstract}

\maketitle

\section{Introduction}
\subsection{Motivations}
The transmission problem in mathematical physics involves interfaces immersed in material bodies that contain two or more components (or inclusions) with distinct physical characters. Any physical processes across the interfaces could be interrupted and loss some continuity. Mathematically, the transmission problems are described by PDEs on each individual component and then the solutions are glued together through the so-called transmission conditions imposed on the interfaces. Such problems cannot be treated as the usual boundary value problems, because the solutions in touching components will interact with each other by the transmission conditions. The main issue for the transmission problem is the regularity of solutions near the interfaces, which is expected to be lower than a usual PDE without transmission, due to its intrinsic nature. We first give two examples of transmission problems that have attracted many interests in history.

\textbf{Example 1:} Traction problem in elasticity. Suppose $\Omega_{+}$ is a bounded Lipschitz domain, $\Omega_- = \R^d \setminus \overline{\Omega_{+}} $ and $S = \partial \Omega_{+}$. A typical traction problem is given by a Lam\'{e} system
\begin{equation}\label{eq.traction}
\left\{
\begin{aligned}
\mu_+ \nabla\cdot (e(u_+)) + \lambda_+ \nabla (\nabla\cdot u_+) & = 0 \qquad \txt{in } \Omega_{+},\\
\mu_- \nabla\cdot (e(u_-)) + \lambda_- \nabla (\nabla\cdot u_-) & = 0 \qquad \txt{in } \Omega_{-}, \\
u_+ - u_- = f \qquad \txt{ on } S, \\
\frac{\partial u_+}{\partial \nu_+} - \frac{\partial u_-}{\partial \nu_-} = g\qquad \txt{ on } S,
\end{aligned}
\right.
\end{equation}
where $e(u) = \nabla u + (\nabla u)^T$ and
\begin{equation*}
\frac{\partial u}{\partial \nu_\pm} = \mu_{\pm} n\cdot e(u) + \lambda_{\pm} n(\nabla\cdot u),
\end{equation*}
with $n$ being the outward normal of $\partial \Omega_{+}$. In (\ref{eq.traction}), we assume that the Lam\'{e} parameters are positive constants satisfying $(\mu_+,\lambda_+) \neq (\mu_-, \lambda_-)$. The third and last equations in (\ref{eq.traction}) are called transmission conditions which glue $u_+$ and $u_-$ together. The traction problem was first introduced by M. Picone in the classical elasticity theory \cite{P54}, with early developments in \cite{St56,C57,LRU66}, etc. Thanks to the method of layer potentials \cite{DKV88}, this problem has been widely studied since 1990s under the assumption that the interface is locally Lipschitz; see \cite{EFV92,ES93,HLM03,EM04,MMS06} and reference therein. The main result in this regime is the estimate of nontangential maximal function of $\nabla u_{\pm}$ in $L^2$ space.

\textbf{Example 2:} Discontinuous source terms. Let $\Omega$ be  bounded domain and $\Omega_{+} \subset \Omega$ and $\Omega_{-} = \Omega \setminus \overline{\Omega_{+}}$. Then the interface between $\Omega_{+}$ and $\Omega_{-}$ is given by $S = \overline{\Omega_{+}} \cap \overline{\Omega_{-}}$. Suppose $A = A_+ \mathbbm{1}_{\Omega_{+}} + A_- \mathbbm{1}_{\Omega_{-}}$, where $A_\pm$ are different constant coefficient matrices and $\mathbbm{1}_E$ denotes the characteristic function on $E$. Also suppose $h = h_+ \mathbbm{1}_{\Omega_{+}} + h_- \mathbbm{1}_{\Omega_{-}}$ and $h_\pm$ are different constant vector in $\R^d$. Hence both $A$ and $h$ are piecewise constant functions with a jump over $S$. Consider the follow elliptic equation
\begin{equation}\label{eq.Disc.Source}
\nabla\cdot (A \nabla u) = \nabla \cdot h \qquad \txt{in } \Omega.
\end{equation}
Simple 1-d examples show that $u$ could be smooth on each individual component (up to the interface), but may not be differentiable on the interface. The literature for (\ref{eq.Disc.Source}) and relevant topics is massive in the past two decades. We will only mention a few works that are closely related to the work of this paper. As far as the author knows, the Schauder theory for (\ref{eq.Disc.Source}) was first studied in \cite{LV00} by Y. Li and M. Vogelius, and then improved in \cite{LN03} by Y. Li and L. Nirenberg. 
The focus of \cite{LV00} and \cite{LN03} was to show that the $C^{1,\alpha}$ estimate in each individual component is independent of the distance between different components, when more than two components are touching. However, they required $\alpha \le \beta/(2(1+\beta))$, if the interface is $C^{1,\beta}$. This restriction was recently relaxed to $\alpha\le \beta/(1+\beta)$ by H. Dong and L. Xu \cite{DX19}. In the case with only two touching components, as in our Example 2, the sharp $C^{1,\alpha}$ estimate with $\alpha = \beta$ has been obtained by H. Dong \cite{Dong12} and J. Xiong and J. Bao \cite{XB13}.

The crucial point we would like to emphasize is that the previous two examples are closely related to each other. By the variational formula (\ref{def.var.sol}),
 (\ref{eq.Disc.Source}) could be equivalently interpreted as a special transmission problem
\begin{equation}\label{eq.TP.dis}
\left\{
\begin{aligned}
\nabla\cdot (A \nabla u_\pm) &= 0 \qquad &\text{ in }& \Omega_\pm, \\
u_+ &= u_- \qquad &\text{ on }&  S,\\
\Big( \frac{\partial u_+}{\partial \nu} \Big)_+ - \Big( \frac{\partial u_-}{\partial \nu} \Big)_- &= - n\cdot h_+ + n\cdot h_- \qquad &\text{ on }& S,
\end{aligned} 
\right.
\end{equation}
where $n$ is the normal direction on $S$ pointing into $\Omega_{-}$ and
\begin{equation*}
\Big( \frac{\partial u}{\partial \nu} \Big)_{\pm} = n\cdot A_{\pm} \nabla u_{\pm}.
\end{equation*}
On the other hand, a transmission problem can also be reduced to a regular problem with discontinuous source terms by a simple trick; see \cite{Dong20}. Consequently, the Schauder theory of the transmission problem may be deduced from the result of \cite{Dong12,XB13}. We mention that before H. Dong's paper \cite{Dong20}, a special case with Laplace operator in both $\Omega_{\pm}$ was considered in \cite{CSS20} by L. A. Caffarelli, M. Soria-Carro and P. R. Stinga.


The main purpose of this paper is two-fold: (i) give a new self-contained proof of the $C^{k,\alpha}$ estimates, independent of recent work \cite{Dong20,CSS20}, for general transmission problem; (ii) use the new idea from (i) to establish the uniform Lipschitz estimate when the coefficients are periodically oscillating at multiple scales.

\subsection{Assumptions and main results}
Let us reformulate the general transmission problem. Suppose $\Omega$ is a bounded Lipschitz domain in $\R^d$ and $\Omega_{\pm}$ be two disjoint subdomains such that $\Omega_{-} = \Omega \setminus \overline{\Omega_{+}}$. Let $S = \overline{\Omega_{+}} \cap \overline{\Omega_{-}}$ be the interface between $\Omega_{+}$ and $\Omega_{-}$. Depending on whether $S$ is connected to $\partial\Omega$, two types of domains might be of interest:

\begin{figure}[h]
	\begin{center}
		\includegraphics[scale =0.3]{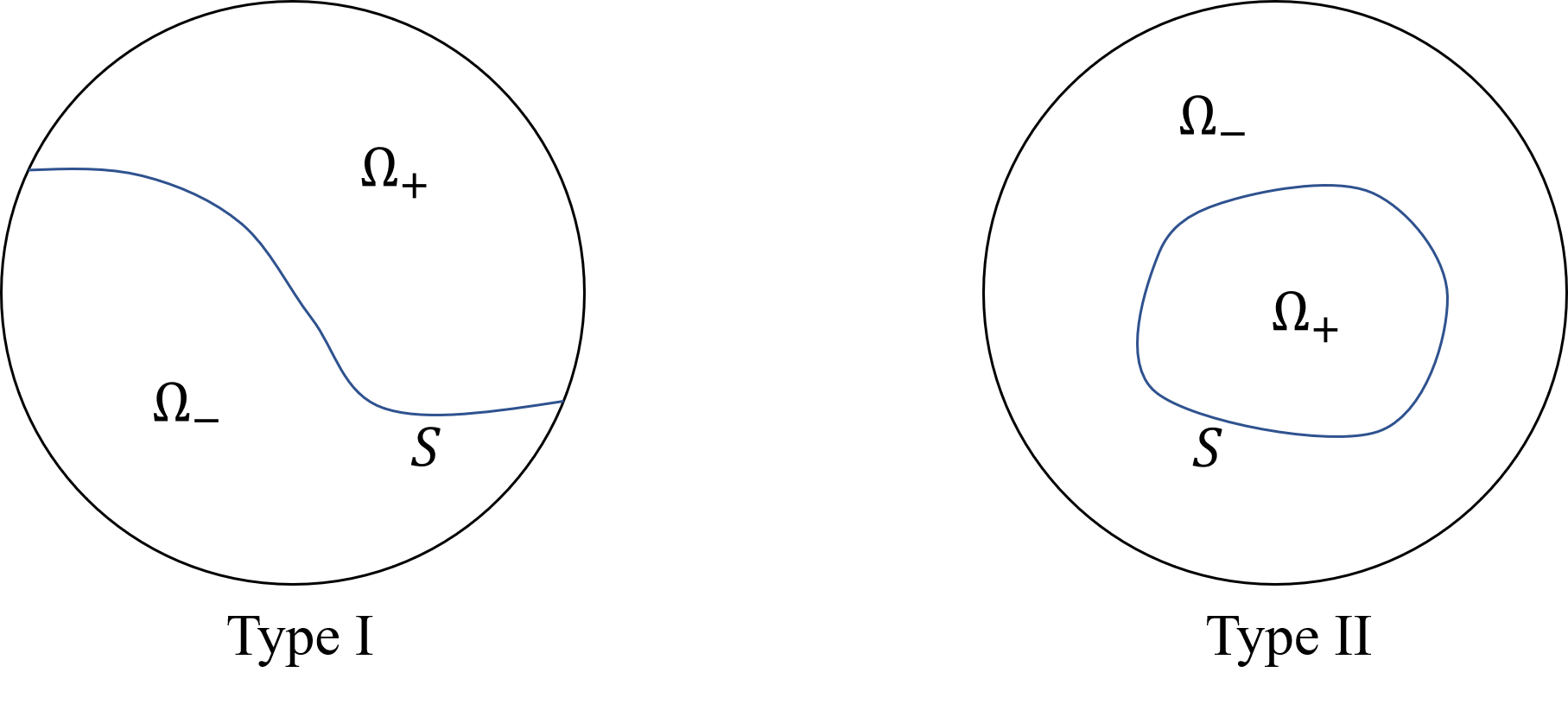}
	\end{center}
	\vspace{-1 em}
	\caption{Two types of domains with two connected components}
\end{figure}

Since we are interested in the local interior estimates (so the geometry of $S\cap \partial \Omega$ will not be considered), the focus of this paper is type I domain. Meanwhile, a type II domain could be decomposed into a union of type I subdomains around $S$. Therefore, throughout this paper, all the domains are assumed to be of type I. 

Let $d\ge 2$ be the dimension of $\Omega$. Consider an elliptic system in $\Omega$ with $m$ equations, which possess distinct behaviours in components $\Omega_{+}$ and $\Omega_{-}$. To characterize such system, we piecewise define the coefficient tensor of system by
\begin{equation}\label{def.A+-}
A(x) = \left\{ 
\begin{aligned}
A_+(x), \qquad \txt{if } x\in \overline{\Omega_+}, \\
A_-(x), \qquad \txt{if } x\in \Omega_-,\\
\end{aligned}
\right.
\end{equation}
where $A_+ = (a_{ij,+}^{\alpha\beta})$ and $A_- = (a_{ij,-}^{\alpha\beta})$, with $1\le i,j\le d, 1\le \alpha,\beta \le m$, are independent coefficient tensors satisfying the strong ellipticity condition: there exists some $\Lambda>0$ such that
\begin{equation}\label{ellipticity}
\Lambda^{-1} |\xi|^2 \le a_{ij,\pm}^{\alpha\beta}(x) \xi_i^\alpha \xi_j^\beta \le \Lambda |\xi|^2, \qquad \txt{for any } x\in \Omega_{\pm} \txt{ and } \xi \in \R^{d\times m},
\end{equation}
where the Einstein summation convention is used here and after.

Let $f$ be a function defined in $\Omega$. For $P \in S$, denote by $f_{\pm}(P)$ the limit of $f(x)$ as $x$ approaching $P$ from $\Omega_{\pm}$ (along the normal direction).
Let $n$ be the normal direction of $S$ pointing into $\Omega_-$. For $P\in S$ and $u\in H^1(\Omega;\R^m)$, define the conormal derivatives on $S$ by
\begin{equation*}
\Big( \frac{\partial u}{\partial \nu} \Big)_{\pm}(P) = n\cdot A_{\pm} (\nabla u)_{\pm}(P).
\end{equation*}

Consider the following transmission problem
\begin{equation}\label{eq.TP}
\left\{
\begin{aligned}
\nabla\cdot (A \nabla u) &= \nabla\cdot h \qquad &\text{ in }& \Omega \setminus S = \Omega_{+} \cup \Omega_{-} \\
u_+ &= u_- \qquad &\text{ on }&  S\\
\Big( \frac{\partial u}{\partial \nu} \Big)_+ - \Big( \frac{\partial u}{\partial \nu} \Big)_- &= g \qquad &\text{ on }& S,
\end{aligned} 
\right.
\end{equation}
where $h:\Omega_{\pm} \mapsto \R^{d\times m}$ is a piecewise defined function given by
\begin{equation}\label{def.h.variable}
h(x) = \left\{ 
\begin{aligned}
&h_+(x), \qquad \txt{if } x \in {\Omega_+}, \\
&h_-(x), \qquad \txt{if } x \in {\Omega_-}.\\
\end{aligned}
\right.
\end{equation}
Throughout this paper, $g$ will be called the transmission data of the solution $u$. Even though we assume the homogenous condition $u_+ = u_-$ on $S$, this is not a restriction at all as we can always add an $A_+$-harmonic function to $u$ in $\Omega_{+}$ to reduce a general non-homogenous case to (\ref{eq.TP}). The notion of the weak solution of (\ref{eq.TP}) will be given precisely in Section 2.

%

Without loss of generality, assume $B_1 = B_1(0)= \Omega$ and $0 \in S$. Let $B_t = B_t(0), B_{t,\pm} = \Omega_{\pm} \cap B_t$ and $S_t = S \cap B_t$. For any subset $E$ in $\R^d$ and $\alpha\in (0,1)$, define
\begin{equation*}
[g]_{C^\alpha(E)} = \sup_{x,y\in E} \frac{|g(x) - g(y)|}{|x-y|^\alpha} \quad \txt{ and } \quad \norm {g}_{C^\alpha(E)} = \norm {g}_{L^\infty(E)} + [g]_{C^\alpha(E)}.
\end{equation*}
Then the $C^{k,\alpha}$ norm may be defined similarly.

For convenience, we will use the following summation notation regarding the subscripts $\pm$: for any objects $a_+$ and $a_-$, define
\begin{equation*}
\sum_{\pm} a_{\pm} = a_+ + a_-.
\end{equation*}
For example, $\sum_{\pm} \norm {f_\pm}_{C^\alpha(\Omega_{\pm})}$ should be interpreted as $\norm {f_+}_{C^\alpha(\Omega_{+})} + \norm {f_-}_{C^\alpha(\Omega_{-})} $.

Now, our first result is a new proof for the Schauder estimates.
\begin{theorem}\label{thm.Cka}
	Let $k\ge 1$ and $\alpha\in (0,1)$. Let $A$ be given by (\ref{def.A+-}) and satisfy (\ref{ellipticity}). Assume $S$ is a $C^{k,\alpha}$ interface, $A_{\pm}\in C^{k-1,\alpha}(\overline{B_{1,\pm}};\R^{d^2\times m^2})$, $h_{\pm} \in C^{k-1,\alpha}(\overline{B_{1,\pm}};\R^{d\times m})$ and $g\in C^{k-1,\alpha}(S_1;\R^m)$. Let $u$ be a weak solution of (\ref{eq.TP}). Then $u\in C^{k,\alpha}(\overline{B_{1/2,\pm}};\R^m)$ and
	\begin{equation*}
	\sum_{\pm} \norm{u}_{C^{k,\alpha}(\overline{B_{1/2,\pm}})} \le C \Big( \sum_{\pm} \norm {h}_{C^{k-1,\alpha}({B_{1,\pm}})} + \norm{g}_{C^{k-1,\alpha}(S_1)} + \norm {u}_{L^2(B_1)} \Big),
	\end{equation*}
	where $C$ depends only on $\Lambda,d,m,k,\alpha$, $C^{k,\alpha}$ character of $S$ and $C^{k-1,\alpha}$ character of $A_{\pm}$.
\end{theorem}


Observe that the above result does not implies the Schauder estimate for the traction problem (\ref{eq.traction}) directly as the transmission condition in (\ref{eq.TP}) is slightly different from that in (\ref{eq.traction}). However, our proof for Theorem \ref{thm.Cka} also works for (\ref{eq.traction}) with the Korn's inequality and obvious modifications.

The second result of this paper is the uniform Lipschitz estimate in periodic homogenization, which is somehow more interesting. The uniform Lipschitz estimate is one of the central problems in homogenization since it is closely related to the estimate of correctors, convergence rates, etc. Let $\Omega, \Omega_{\pm}$ and $S$ be the same as before. Let $\e = \{ \e_+,\e_-\}$ be a pair of two small parameters with $\e_{\pm} \in (0,1)$. Suppose that the coefficient tensor is given by
\begin{equation}
A^\e(x) = \left\{ 
\begin{aligned}
A_+^{\e_+}(x) &= A_+(x/\e_+), \qquad \txt{if } x \in \overline{\Omega_+}, \\
A_-^{\e_-}(x) &= A_-(x/\e_-), \qquad \txt{if } x \in \Omega_-.\\
\end{aligned}
\right.
\end{equation}
We assume that $A_{\pm}$ is a $Y_{\pm}$-periodic coefficient tensor in $\R^d$ satisfying (\ref{ellipticity}), where $Y_+$ and $Y_-$ are completely independent parallelotopes. Also, for any $P\in S$, define the conormal derivatives by
\begin{equation*}
\Big( \frac{\partial u}{\partial \nu_\e} \Big)_{\pm}(P) = n\cdot A^{\e_\pm}_{\pm}(P) (\nabla u)_{\pm}(P).
\end{equation*}
Consider the following transmission problem with oscillating coefficient tensor $A^\e$
\begin{equation}\label{eq.TP.e}
\left\{
\begin{aligned}
\nabla\cdot (A^\e \nabla u^\e) &= \nabla\cdot h \qquad &\text{ in }& \Omega \setminus S, \\
u^\e_+ &= u^\e_- \qquad &\text{ on }&  S,\\
\Big( \frac{\partial u^\e}{\partial \nu_\e} \Big)_+ - \Big( \frac{\partial u^\e}{\partial \nu_\e} \Big)_- &= g \qquad &\text{ on }& S,
\end{aligned} 
\right.
\end{equation}
where $h$ is given by (\ref{def.h.variable}).

\begin{theorem}\label{thm.Lip}
	Let $\alpha\in (0,1)$. Assume $A_{\pm}$ is $Y_{\pm}$-periodic and satisfies (\ref{ellipticity}), $S$ is a $C^{1,\alpha}$ interface, $h_{\pm} \in C^{\alpha}(\overline{B_{1,\pm}};\R^{d\times m})$ and $g\in C^{\alpha}(S_1;\R^m)$. Let $\e  = \{ \e_+,\e_-\}$ and $0<\e_+ \le \e_- < 1$. Let $u^\e$ be a weak solution of (\ref{eq.TP.e}). Then 
	
	\begin{itemize}
		\item[(i)] For every $r\in (\e_-,1)$,
		\begin{equation}\label{est.Lip}
		\bigg( \fint_{B_r} |\nabla u^\e|^2 \bigg)^{1/2} \le C\bigg[ \bigg( \fint_{B_1} |\nabla u^\e|^2 \bigg)^{1/2} + \sum_{\pm} [h]_{C^\alpha(\Omega_{\pm} \cap B_1)}  + \norm {g}_{C^\alpha(S\cap B_1)}\bigg],
		\end{equation}
		where $C$ depends only on $\Lambda,d,m,\alpha$ and the $C^{1,\alpha}$ character of $S$.
		
		\item[(ii)] If, in addition, $A_-$ is $C^\alpha$-H\"{o}lder continuous, then (\ref{est.Lip}) holds for all $r\in (\e_+,1)$ with constant $C$ depending additionally on the $C^\alpha$ character of $A_-$.
		
		\item[(iii)] If both $A_\pm$ are both $C^\alpha$-H\"{o}lder continuous, then
		\begin{equation*}
		\norm {\nabla u^\e}_{L^\infty(B_{1/2})} \le C\bigg[ \bigg( \fint_{B_1} |\nabla u^\e|^2 \bigg)^{1/2} + \sum_{\pm} [h]_{C^\alpha(\Omega_{\pm} \cap B_1)}  + \norm {g}_{C^\alpha(S\cap B_1)}\bigg],
		\end{equation*}
		where $C$ depends additionally on the $C^\alpha$ characters of $A_\pm$.
	\end{itemize}
\end{theorem}

Theorem \ref{thm.Lip} is a new result for transmission problem. Even in the simple case without transmission (i.e., $h = 0$ and $g = 0$), our result is still new and extends the recent result \cite{Jo19} by M. Josien and \cite{JR19} by M. Josien and C. Raithel to the most general setting in periodic homogenization (also see related work \cite{BLL15}). Recall that in \cite{Jo19} the large-scale Lipschitz estimate was established under very restrictive conditions that $S$ is flat and the periodic structures of $A_+$ and $A_-$ are of the same scale (i.e., $\e_+ = \e_-$) and parallel to $S$ with a common period in the parallel directions.
The structure condition on $A_\pm$ then was removed in \cite{JR19}; but the interface remains to be flat.
In our theorem, these restrictions are all removed, i.e., the structures of $A_+$ and $A_-$ are completely independent and oscillating at different microscopic scales (i.e., $\e_{+}\neq \e_-$), and $S$ is an arbitrary $C^{1,\alpha}$ interface. We should emphasize that, unlike some other multi-scale problems \cite{GS20,NSX19,NS20}, we have no restriction on the ratio $\e_+/\e_-$.

We point out that Theorem \ref{thm.Lip} is sharp in several aspects. First, the Lipschitz estimate is the best regularity one can expect on the interface for the transmission problem. It is also the optimal uniform regularity in homogenization theory if no extra correction is introduced. Second, the Lipschitz estimate in Theorem \ref{thm.Lip} is stated at three different levels in terms of the smoothness of $A_-$ and $A_+$. We point out that the different ranges of $r$ at these levels are all optimal. Finally, the regularity assumptions for the interface $S$ and the transmission data are optimal in the sense that the result may fail if $\alpha = 0$ for any of them.

\subsection{Ideas of the proof}
Our idea of proving Theorem \ref{thm.Cka} is different from \cite{CSS20} or \cite{Dong20}. As noticed in \cite[Remark 4.5]{CSS20}, their approach is based on the mean value property of harmonic functions which is unable to be extended to equations/systems with variable and discontinuous coefficients. In \cite{Dong20}, the author used a simple trick to reduce the transmission problem to a problem with discontinuous source terms and therefore the Schauder estimate follows readily from the existing result in \cite{Dong12,XB13}. In this paper, we give a new self-contained proof of Theorem \ref{thm.Cka}, which also inspires the proof of Theorem \ref{thm.Lip}. Our approach to break the barrier is the full use of the so called ``piecewise linear solutions'' given in Definition \ref{def.Lset}. It is a natural generalization of the linear solutions arising in transmission problems. Taking advantage of the ``piecewise linear solutions'', along with a modified Campanato iteration argument, we are able to give a complete and clean proof for Theorem \ref{thm.Cka}.

The proof of Theorem \ref{thm.Lip} is based on a modified Campanato iteration method developed recently in, e.g., \cite{AS16,Sh17,Sh18,AKM19} etc. This method roughly contains two parts: a quantitative convergence rate and an excess decay estimate. Since our transmission problem has two microscopic scales, $\e_+$ and $\e_-$, the convergence rate and the excess decay estimate differs according to the relationship between them. Precisely, if $0<\e_+< \e_- < 1$, $u^\e$ converges to a two-sided homogenized solution $u^0$ and for some absolute constant $\sigma>0$,
\begin{equation}\label{rate1}
\norm {u^\e - u^0} \lesssim \e_-^\sigma.
\end{equation}
In this case, the system homogenizes in both $\Omega_+$ and $\Omega_{-}$. However, if $A_-$ is $C^\alpha$-H\"{o}lder continuous and $\e_+ < 1 < \e_-$ (this will happen if a mesoscopic problem is rescaled), then $u^\e$ converges to a one-sided homogenized solution $\bar{u}$ and
\begin{equation}\label{rate2}
\norm {u^\e - \bar{u}} \lesssim \e_+^\sigma + \e_-^{-\alpha}.
\end{equation}
In this case, the system homogenizes in $\Omega_{+}$ and exhibits a blow-up behavior in $\Omega_{-}$.
The convergence rates in (\ref{rate1}) and (\ref{rate2}) go a long way in explaining why the Lipschitz estimate in Theorem \ref{thm.Lip} has three different scales. The proofs of (\ref{rate1}) and (\ref{rate2}) follow from the idea of \cite{SZ17} by a duality argument, which requires no a priori estimate of $u^0$ or $\bar{u}$ (except for the energy estimate).

The key point in the second part of the proof is to construct an appropriate quantity (excess) to measure the flatness of $u^\e$ at mesoscopic scales. In a transmission problem, the correct quantity is
inspired by the proof of Theorem \ref{thm.Cka}, which involves the ``piecewise linear solutions'', i.e.,
\begin{equation*}
\begin{aligned}
H(u;t) := \inf\limits_{\substack{\ell \in \mathscr{L}}}  \bigg\{ \frac{1}{t}   \bigg( \fint_{B_t} |u - \ell |^2 \bigg)^{1/2} + \Norm {g - \mathscr{T}(\ell) }_{L^\infty(S_t)} \bigg\},
\end{aligned}
\end{equation*}
where $\mathscr{L}$ is the set of all the ``piecewise linear solutions'' and $\mathscr{T}(\ell)$ is the transmission data of the specific $\ell \in \mathscr{L}$ (see Definition \ref{def.Lset}). Once we have the correct excess, following the elegant framework of Z. Shen in \cite{Sh17}, the desired estimate is a combination of the convergence rates, an interface stability lemma and an iteration lemma. We should point out that the large-scale Lipschitz estimate for a simpler interface problem (without transmission, i.e., $g = 0, h=0$) has been proved in \cite{JR19} in the random setting. The excess defined in \cite{JR19} involves $\nabla u$ and thus the ``generalized correctors" (of which the sublinear growth is the emphasis of the paper). However, we only uses the usual correctors for $A_{+}$ and $A_-$ in the proof of convergence rates and the ``generalized correctors'' are avoided. Notice that in our multi-scale situation with $\e_+ \neq \e_-$, one will encounter new problems in defining the  ``generalized correctors''. Of course, our proof of Theorem \ref{thm.Lip} can be naturally extended to the stationary and ergodic regime.

The organization of this paper is as follows. The definition of weak solution is given in Sectoin 2. The proof of Theorem \ref{thm.Cka} is given in Section 3. In Section 4, we prove the convergence rates in the setting of periodic homogenization. The uniform Lipschitz estimate in Theorem \ref{thm.Lip} is proved in Section 5.

{\bf Acknowledgement.} The author would like to thank Prof. Hongjie Dong for pointing out the closely related work \cite{Dong12,DX19,Dong20} and many helpful discussions, after the author finished the draft of the paper.

\section{Well-posedness}
In this section, we give a definition for the weak solution of the transmission problem and prove the well-posedness (existence and uniqueness) of the transmission problem with Dirichlet boundary condition.

Let $\Omega$ be a Lipschitz domain and $S$ a Lipschitz interface that separates $\Omega$ into two Lipschitz components $\Omega_{+}$ and $\Omega_{-}$. Let $A$ satisfy (\ref{def.A+-}) and (\ref{ellipticity}). Suppose 
\begin{equation}\label{def.F}
F(x) = \left\{ 
\begin{aligned}
&F_+(x), \qquad \txt{if } x \in {\Omega_+}, \\
&F_-(x), \qquad \txt{if } x \in {\Omega_-},\\
\end{aligned}
\right.
\qquad
h(x) = \left\{ 
\begin{aligned}
&h_+(x), \qquad \txt{if } x \in {\Omega_+}, \\
&h_-(x), \qquad \txt{if } x \in {\Omega_-},\\
\end{aligned}
\right.
\end{equation}
with $F_{\pm} \in L^2(\Omega_{\pm};\R^m)$ and $h_{\pm} \in L^2(\Omega_{\pm};\R^{d\times m})$. Suppose that $g\in H^{-1/2}(S;\R^m)$. We say $u\in H^1(\Omega;\R^m)$ is a weak solution of
\begin{equation}\label{eq.TP.F}
\left\{
\begin{aligned}
\nabla\cdot (A \nabla u) &= F + \nabla\cdot h \qquad &\text{ in }& \Omega \setminus S = \Omega_{+} \cup \Omega_{-}, \\
u_+ &= u_- \qquad &\text{ on }&  S,\\
\Big( \frac{\partial u}{\partial \nu} \Big)_+ - \Big( \frac{\partial u}{\partial \nu} \Big)_- &= g \qquad &\text{ on }& S,
\end{aligned} 
\right.
\end{equation}
if for any $\phi \in H^1_0(\Omega;\R^m)$,
\begin{equation}\label{def.var.sol}
\int_{\Omega} A\nabla u\cdot \nabla \phi = \int_S (g-n\cdot h_+ + n\cdot h_-) \cdot \phi d\sigma - \sum_{\pm} \int_{\Omega_{\pm}} F_\pm \cdot \phi + \sum_{\pm} \int_{\Omega_{\pm}} h \cdot \nabla \phi,
\end{equation}
where $d\sigma = dH^{d-1}|_S$ is the surface measure. To guarantee that the first integral on the right-hand side of (\ref{def.var.sol}) is well-defined, we require some necessary condition on $h_\pm$ so that the $n\cdot (h_+ - h_-) \in H^{-1/2}(S;\R^m)$, as the trace theorem implies $\phi|_S \in H^{1/2}(S;\R^m)$. Fortunately, this is not going to be a problem as we will always assume $h_\pm$ are continuous in this paper.


\begin{remark}\label{rmk.specialcase}
	Two particular cases are worth being pointed out:
	\begin{itemize}
		\item If $F = 0$ and $h = 0$, then in view of (\ref{def.var.sol}), we may reformulate (\ref{eq.TP.F}) as
		\begin{equation}\label{eq.gdsigma}
		-\nabla\cdot (A \nabla u)  =  g d\sigma \quad  \txt{in } \Omega.
		\end{equation}
		
		\item If $F = 0$ and $g = n\cdot h_+ - n\cdot h_-$ on $S$, then we may reformulate (\ref{eq.TP.F}) as
		\begin{equation}\label{def.eq.Dh}
		\nabla\cdot (A \nabla u)  = \nabla\cdot h \quad \txt{in } \Omega.
		\end{equation}
		Therefore, all the estimates in this paper apply to system (\ref{def.eq.Dh}).
	\end{itemize}
\end{remark}


We end up this section with the well-posedness of the Dirichlet problem which will ensure the existence and uniqueness of the transmission problem throughout this paper.
\begin{theorem}
	Let $\Omega$ and $S$ be as above. Suppose $A$ satisfies (\ref{def.A+-}) and (\ref{ellipticity}), and let $F,h$ and $g$ satisfy the above assumptions. Then the system (\ref{eq.TP.F}), subject to the boundary condition $u = f \in H^{1/2}(\partial \Omega;\R^m)$ on $\partial \Omega$, has a unique solution $u\in H^1(\Omega;\R^m)$ such that
	\begin{equation*}
	\norm {u}_{H^1(\Omega)} \le C\big( \norm {f}_{H^{1/2}(\partial\Omega)} + \norm {g-n\cdot h_+ + n\cdot h_-}_{H^{-1/2}(S)} + \norm {F}_{L^2(\Omega)} + \norm {h}_{L^2(\Omega)} \big).
	\end{equation*}
\end{theorem}
\begin{proof}
	By (\ref{def.var.sol}) and the Lax-Milgram theorem, this can be proved by a standard argument. 
\end{proof}

\section{Schauder theory}
In this section, we develop a new approach to establish the Schauder theory for the transmission problem.

\subsection{Piecewise linear solutions}
The main difficulty in the transmission problem is that, even if the interface is flat, the set of all linear functions does not contain all the ``linear solutions'' or ``plane-like solutions'' satisfying a basic transmission condition. If the interface $S$ is curved, it is unclear how we can construct simple functions to approximate the solution. To overcome this difficulty, we first introduce a family of ``piecewise linear solutions'' for the interface transmission problem in half spaces.

Let $e_k = (0,\cdots, 1,\cdots, 0)\in \R^d$ with $1$ in the $k$th position. For convenience, let $\{ x_d = 0\} = \{ x = (x',x_d) \in \R^d: x_d = 0 \}$.
\begin{proposition}\label{prop.jump}
	Let $M_{\pm}$ be constant matrices in $ \R^{m\times d}$. The following statements are equivalent:
	\begin{itemize}
		\item[(i)] For all $x\in \{x_d = 0\}$, $M_+ x = M_- x$;
		
		\item[(ii)] There exists some $q\in \R^m$ such that
		\begin{equation*}
		M_+ - M_- = q\otimes e_d.
		\end{equation*}
		
		\item[(iii)] It holds
		\begin{equation}\label{eq.jump3}
		(I -e_d\otimes e_d) M_+^* = (I -e_d\otimes e_d) M_-^*.
		\end{equation} 
	\end{itemize}
\end{proposition}
\begin{proof}
	(ii) $\Rightarrow$ (i): For any $x \in \{ x_d = 0 \}$, one may write $x = \sum_{k = 1}^{d-1} x_k e_k$. Then,
	\begin{equation*}
	(M_+ - M_-)x = q\otimes e_d \Big(\sum_{k = 1}^{d-1} x_k e_k \Big) = 0.
	\end{equation*}
	
	(iii) $\Rightarrow$ (ii): Suppose
	\begin{equation*}
	M_+ - M_- = \sum_{k = 1}^d q_k \otimes e_k, \qquad q_k \in \R^m.
	\end{equation*}
	If (\ref{eq.jump3}) holds, then
	\begin{equation*}
	0 = (I -e_d\otimes e_d) (M_+^* - M_-^*) = (I -e_d\otimes e_d) \Big( \sum_{k = 1}^d  e_k \otimes  q_k \Big) = \sum_{k = 1}^{d-1} e_k \otimes q_k.
	\end{equation*}
	Hence, $q_k = 0$ for $k = 1,2,\cdots,d-1$. This implies $M_+ - M_- = q_d \otimes e_d$ for some $q_d \in \R^m$.
	
	(i) $\Rightarrow$ (iii): Note that for any $x\in \R^d$, $(I- e_d \otimes e_d) x \in \{ x_d = 0\}$. Thus, (i) implies
	\begin{equation*}
	M_+ (I- e_d \otimes e_d) x = M_- (I- e_d \otimes e_d) x, \qquad \txt{for any } x\in \R^d.
	\end{equation*}
	This implies (iii).
\end{proof}


Suppose $M_{\pm}$ satisfies the condition in Proposition \ref{prop.jump}. Let $\ell(x):= M_+ x \mathbbm{1}_{\{x_d \ge 0\}} + M_- x \mathbbm{1}_{\{x_d<0\}}$. If $A_{\pm}$ are constant, it is easy to verify that
\begin{equation}\label{eq.TP.linear}
\left\{
\begin{aligned}
\nabla\cdot (A \nabla \ell) &= 0 \qquad &\text{ in }& \R^d \setminus \{x_d = 0\}, \\
\ell_+ &= \ell_- \qquad &\text{ on }&  \{x_d = 0\},\\
\bigg( \frac{\partial \ell}{\partial \nu} \bigg)_+ - \bigg( \frac{\partial \ell}{\partial \nu} \bigg)_- &= e_d\cdot A_+ M_+^* - e_d\cdot A_- M_-^* \qquad &\text{ on }& \{x_d = 0\}.
\end{aligned} 
\right.
\end{equation}
The second equation in (\ref{eq.TP.linear}) holds because of Proposition \ref{prop.jump} (i). The above fact motivates us to define the ``piecewise linear solutions''.

\begin{definition}\label{def.Lset}
	We say $\ell(x):\R^d \mapsto \R^m$ is a piecewise linear solution of the interface transmission problem (\ref{eq.TP.linear}), if $\ell(x) = M_+ x \mathbbm{1}_{\{x_d \ge 0\}} + M_- x \mathbbm{1}_{\{x_d<0\}} + q$, where $q\in \R^m$ and $M_\pm$ satisfies the condition of Proposition \ref{prop.jump}. Denote by $\mathscr{L}$ the space of all the piecewise linear solutions of (\ref{eq.TP.linear}). For any $\ell \in \mathscr{L}$, define
	\begin{equation*}
	\mathscr{T}(\ell) := \Big( \frac{\partial \ell}{\partial \nu} \Big)_+ - \Big( \frac{\partial \ell}{\partial \nu} \Big)_- = e_d\cdot A_+ M_+^* - e_d\cdot A_- M_-^*.
	\end{equation*}
\end{definition}

Note that the set $\mathscr{L}$ is independent of the tensors $A_\pm$. But the liner map $\mathscr{T}:\mathscr{L}\mapsto \R^m$ depends on $A_{\pm}$.

\subsection{A basic $C^{1,\alpha}$ estimate}
The key to the proof of Theorem \ref{thm.Cka} is a basic $C^{1,\alpha}$ estimate in the case that $A$ is piecewise constant and $S$ is flat. Our approach relies on the notion of piecewise linear solutions and a fundamental result contained in the following lemma.
\begin{lemma}\label{lem.basic}
	Let $\Omega_{\pm} = \R^d_{\pm} = \{x\in \R^d: \pm x_d > 0 \}$ and $S = \{x_d = 0\}$. Let $A$ be defined by (\ref{def.A+-}) and $A_{\pm}$ is constant. Suppose that $u$ is a weak solution of $\nabla\cdot (A\nabla u) = 0$ in $B_1 = B_1(0)$. Then for any $k\ge 1$,
	\begin{equation*}
	\sum_{\pm} \norm {u}_{H^k(B_{1/2,\pm})} \le C_k \norm {u}_{L^2(B_1)}.
	\end{equation*}
\end{lemma}
The above estimate should be well-known. However, since we could not find a straightforward proof, we would like to include the outline of a proof here. In fact, since $A$ remains constant in directions parallel to $S = \{ x_d = 0 \}$, one may take arbitrary derivatives to the system in $x'=(x_1,x_2,\cdots,x_{d-1})$. Then by the Caccioppli inequality and induction, we see that for any $k\ge 0$, $\norm {\nabla_{x'}^k \nabla u}_{L^2(B_1/2)} \le C\norm {u}_{L^2(B_1)}$. To recover the derivatives in $x_d$, we may rearrange the equation $\nabla\cdot (A\nabla u) = 0$ as, in either $B_{1/2,+}$ or $B_{1/2,-}$,
\begin{equation}\label{eq.Dd2u}
\frac{\partial^2 u}{\partial x_d^2} = \txt{a linear combination of } \frac{\partial^2 u}{\partial x_i \partial x_j} \txt{ or } \frac{\partial u}{\partial x_j} \txt{ with } 1\le i \le d-1 \txt{ and } 1\le j \le d.
\end{equation}
This implies $\sum_{\pm} \norm {\nabla^2 u}_{L^2(B_{1/2,\pm})} \le C\norm {u}_{L^2(B_1)}$. By taking derivatives in $x_d$ to the equation (\ref{eq.Dd2u}) and induction, we derive the estimate for arbitrarily higher-order derivatives.

Observe that Lemma \ref{lem.basic} implies that $u$ is smooth in either $\overline{ B_{1/2,+}}$ or $\overline{ B_{1/2,-}}$, namely,
\begin{equation}\label{est.basic.Cka}
\sum_{\pm} \norm {u}_{C^{k,\alpha}(\overline{ B_{1/2,\pm}})} \le C_{k,\alpha} \norm {u}_{L^2(B_1)},
\end{equation}
for any $k\ge 0, \alpha \in (0,1]$.
In particular, this implies that $u$ is Lipschitz continuous in $B_{1/2}$, which is optimal, as simple examples constructed from (\ref{eq.TP.linear}) show that $u$ may not be differentiable on the interface $\{x_d = 0\}$.

\begin{lemma}\label{lem.g0}
	Let $g_0 \in \R^m$ be constant. Let $u$ be a weak solution of
	\begin{equation}\label{eq.g0.const}
	\left\{
	\begin{aligned}
	\nabla\cdot (A \nabla u) &= 0 \qquad &\text{ in }& B_1 \setminus \{x_d = 0\}, \\
	u_+ &= u_- \qquad &\text{ on }&  B_1\cap \{x_d = 0\},\\
	\Big( \frac{\partial u}{\partial \nu} \Big)_+ - \Big( \frac{\partial u}{\partial \nu} \Big)_- &= g_0 \qquad &\text{ on }& B_1\cap \{x_d = 0\}.
	\end{aligned} 
	\right.
	\end{equation}
	Then, for any $k\ge 0$ and $\alpha \in (0,1]$,
	\begin{equation*}
	\sum_{\pm} \norm {u}_{C^{k,\alpha}(\overline{ B_{1/2,\pm}})} \le C_{k,\alpha} \big( \norm {u}_{L^2(B_1)} + |g_0| \big).
	\end{equation*}
\end{lemma}

\begin{proof}
	The key insight of proof is that we can find a piecewise linear solution so that the transmission data equals $g_0$ (thus $\mathscr{T}: \mathscr{L}\mapsto \R^m$ is surjective). In other words, we would like to find $M_{\pm}$ satisfying the condition in Proposition \ref{prop.jump} so that $\ell(x) = M_+ x \mathbbm{1}_{\{x_d \ge 0\}} + M_- x \mathbbm{1}_{\{x_d<0\}} \in \mathscr{L}$ and
	\begin{equation*}
	\mathscr{T}(\ell) = e_d\cdot A_+ M_+^* - e_d\cdot A_- M_-^* = g_0.
	\end{equation*}
	The construction of such $\ell$ is not unique. For simplicity, let $M_- = 0$. Then, from Proposition \ref{prop.jump} (ii), $M_+ = q\otimes e_d$ for some $q\in \R^m$. Therefore,
	\begin{equation*}
	\mathscr{T}(\ell) = e_d\cdot A_+ M_+^* = (e_d\cdot A_+ e_d) q = g_0.
	\end{equation*}
	Note that, due to (\ref{ellipticity}), $(e_d\cdot A_+ e_d) = (a_{dd,+}^{\alpha\beta})_{1\le \alpha,\beta\le m}$ is a constant positive definite $m\times m$ matrix. Hence, we obtain $q = (e_d\cdot A_+ e_d)^{-1} g_0$. It follows that
	\begin{equation}
	\ell(x) =  [(e_d\cdot A_+ e_d)^{-1} g_0 ] x_d
	\end{equation}
	is a weak solution of (\ref{eq.g0.const}). Clearly,
	\begin{equation}\label{est.ell.g0}
	\sum_{\pm} \norm {\ell }_{C^{k,\alpha}(\overline{ B_{1/2,\pm}})} \le C|g_0|.
	\end{equation}
	Next, consider $w = u - \ell$. Observe (from (\ref{def.eq.Dh}) with $h = 0$) that $w$ is a weak solution of $\nabla\cdot (A\nabla w) = 0$ in $B_1$. By (\ref{est.basic.Cka}), we have
	\begin{equation*}
	\sum_{\pm} \norm {w}_{C^{k,\alpha}(\overline{ B_{1/2,\pm}})} \le C_{k,\alpha} \norm {w}_{L^2(B_1)},
	\end{equation*}
	which, along with (\ref{est.ell.g0}), gives the desired estimate.
\end{proof}

\begin{lemma}\label{lem.v-lv}
	Let $v$ be a weak solution of
	\begin{equation}\label{eq.TP.g0}
	\left\{
	\begin{aligned}
	\nabla\cdot (A \nabla v) &= 0 \qquad &\text{ in }& B_t \setminus \{x_d = 0\}, \\
	v_+ &= v_- \qquad &\text{ on }&  B_t\cap \{x_d = 0\},\\
	\Big( \frac{\partial v}{\partial \nu} \Big)_+ - \Big( \frac{\partial v}{\partial \nu} \Big)_- &= g_0 \qquad &\text{ on }& B_t\cap \{x_d = 0\}.
	\end{aligned} 
	\right.
	\end{equation}
	Let $M_{\pm} = (\nabla v)^*_\pm (0)$ and $q = v(0) \in \R^m$. Then $\ell_v(x) = M_+ x \mathbbm{1}_{\{x_d \ge 0\}} + M_- x \mathbbm{1}_{\{x_d<0\}} + q \in \mathscr{L}$. Moreover,
	for any $x\in B_{t/2}$, we have
	\begin{equation}\label{est.v-lv}
	|v(x) - \ell_v(x)| \le C(|x|/t)^{1+\beta} \bigg\{ \bigg( \fint_{B_t} |v|^2 \bigg)^{1/2}  + t|g_0| \bigg\},
	\end{equation}
	where $\beta \in (0,1)$.
\end{lemma}
\begin{proof}
	By rescaling, it suffices to prove the case $t = 1$. First of all, Lemma \ref{lem.basic} implies that $v$ is smooth up to the the interface and thus $v(0)$ and $\nabla v(0)$ are well-defined.
	Since $v_+ = v_-$ on $S_t$, the tangential derivatives of $v_+$ and $v_-$ coincides, namely, $(1-e_d\otimes e_d) (\nabla v)_+ = (1-e_d\otimes e_d) (\nabla v)_-$ on $B_t\cap \{x_d = 0 \}$. Therefore, $M_\pm$ satisfies Proposition \ref{prop.jump} (iii) and thus $\ell_v \in \mathscr{L}$, by Definition \ref{def.Lset}.
	
	Moreover, for $x\in B_{1/2}$,
	\begin{equation}\label{est.w1}
	\begin{aligned}
	|v(x) - \ell_v(x)| &= |v(x) - v(0) - x\cdot (\nabla v)_{\pm}(0)|\\
	& \le C|x|^{1+\beta} [v]_{C^{1,\beta}(B_{1/2},\pm)} \\
	& \le C|x|^{1+\beta} \big( \norm {v}_{L^2(B_1)}  + |g_0| \big),
	\end{aligned}
	\end{equation}
	where we have used Lemma \ref{lem.g0} (with $k = 1$) in the last inequality. This is exactly the desired estimate with $t = 1$.
\end{proof}

The following lemma is a version of the excess decay estimate in Schauder theory. This type of estimate is crucial in studying $C^{1,\alpha}$ estimate, as well as the Lipschitz estimate in homogenization; see Lemma \ref{lem.vt}.
\begin{lemma}\label{lem.C1a.excess}
	Let $v$ be a solution of (\ref{eq.TP.g0}) and $\beta \in (0,1)$. Then, for every $\rho \in (0,t)$,
	\begin{equation}\label{est.inf.L}
	\begin{aligned}
	& \inf_{\ell \in \mathscr{L}} \bigg\{ \bigg( \fint_{B_{\rho}} |v - \ell|^2 \bigg)^{1/2} + \rho|g_0 - \mathscr{T}(\ell)| \bigg\} \\
	& \qquad \le C(\rho/t)^{1+\beta} \inf_{\ell \in \mathscr{L}} \bigg\{ \bigg( \fint_{B_{t}} |v - \ell|^2 \bigg)^{1/2} + t|g_0 - \mathscr{T}(\ell) |\bigg\}.
	\end{aligned}
	\end{equation}
\end{lemma}
\begin{proof}
	Let $\ell_v$ be given as in Lemma \ref{lem.v-lv}, then $g_0 = \mathscr{T}(\ell_v)$. Hence, (\ref{est.v-lv}) yields
	\begin{equation}\label{est.vBrho}
	\bigg( \fint_{B_{\rho}} |v - \ell_v|^2 \bigg)^{1/2} + \rho|g_0 - \mathscr{T}(\ell_v)| \le C(\rho/t)^{1+\beta} \bigg\{ \bigg( \fint_{B_{t}} |v|^2 \bigg)^{1/2} + t|g_0|\bigg\}.
	\end{equation}
	Now, since any $\ell \in \mathscr{L}$ is a weak solution of (\ref{eq.TP.linear}), $v-\ell$ is also a weak solution with the transmission data equalling $g_0 -\mathscr{T}(\ell)$. Applying the estimate (\ref{est.vBrho}) to $v - \ell$, we have
	\begin{equation}\label{est.v-l}
	\bigg( \fint_{B_{\rho}} |v - \ell_v|^2 \bigg)^{1/2} + \rho|g_0 - \mathscr{T}(\ell_v)| \le C(\rho/t)^{1+\beta} \bigg\{ \bigg( \fint_{B_{t}} |v - \ell|^2 \bigg)^{1/2} + t|g_0 - \ell|\bigg\},
	\end{equation}
	for any $\ell \in \mathscr{L}$. Finally, by taking infimum of the right-hand side of (\ref{est.v-l}) over all $\ell\in \mathscr{L}$, we obtain (\ref{est.inf.L}).
\end{proof}


\subsection{Perturbation and iteration}
In this subsection, we prove the $C^{1,\alpha}$ estimate for the general transmission problem by the method of perturbation and Campanato iteration. The following well-known lemma will be useful.
\begin{lemma}[e.g., \cite{GM12}]\label{lem.Iteraton}
	Let $\Phi: \R_+ \mapsto \R_+$ be a non-negative and non-decreasing function satisfying
	\begin{equation*}
	\Phi(\rho) \le C (\rho/r)^\beta \Phi(r) + Br^\alpha,
	\end{equation*}
	for some $\beta>\alpha>0, C>0$ and for all $0<\rho<r<R$, where $R>0$ is given. Then
	there is $C_1$ depending only on $C,\alpha$ and $\beta$ so that
	\begin{equation*}
	\Phi(\rho) \le C_1 (\rho/r)^\alpha \Phi(r) + B\rho^\alpha.
	\end{equation*}
\end{lemma}

\begin{theorem}\label{lem.const.C1a}
	Let $A_{\pm}$ be $C^\alpha$-H\"{o}lder continuous. Let $S_t = B_t\cap \{x_d = 0\}$. Suppose $u$ is a weak solution of
	\begin{equation}\label{eq.TP.C1a}
	\left\{
	\begin{aligned}
	\nabla\cdot (A \nabla u) &= \nabla\cdot h \qquad &\text{ in }& B_1 \setminus S_1, \\
	u_+ &= u_- \qquad &\text{ on }&  S_1, \\
	\Big( \frac{\partial u}{\partial \nu} \Big)_+ - \Big( \frac{\partial u}{\partial \nu} \Big)_- &= g \qquad &\text{ on }& S_1,
	\end{aligned} 
	\right.
	\end{equation}
	where $g\in C^\alpha(S_1,\R^m)$,
	\begin{equation}\label{def.h.var}
	h(x) = \left\{ 
	\begin{aligned}
	&h_+(x), \qquad \txt{if } x \in {B_{1,+}}, \\
	&h_-(x), \qquad \txt{if } x \in {B_{1,-}},\\
	\end{aligned}
	\right.
	\end{equation}
	and $h_{\pm} \in C^\alpha(\overline{B_{1,\pm}};\R^{d\times m})$. Then 
	\begin{equation*}
	\sum_{\pm} \norm {u}_{C^{1,\alpha}(\overline{B_{1/2,\pm}})} \le C\Big( \sum_{\pm} \norm {h}_{C^\alpha(B_{1,\pm})} + \norm {g}_{C^\alpha(S_1)} + \norm {u}_{L^2(B_1)} \Big).
	\end{equation*}
\end{theorem}
\begin{proof}
	Note that it suffices to show that for every $\rho \in (0,1/4)$
	\begin{equation}\label{est.u-l}
	\inf_{\ell \in \mathscr{L}} \bigg( \fint_{B_{\rho}} |u - \ell|^2 \bigg)^{1/2} \le C\rho^{1+\alpha} \Big( \sum_{\pm} \norm {h}_{C^\alpha(B_{3/4,\pm})} + \norm {g}_{C^\alpha(S_{3/4})}+ \norm {u}_{L^2(B_{3/4})} \Big),
	\end{equation}
	where $B_\rho$'s are arbitrary balls centered on $S_{1/2}$. Then, this implies that $u$ is $C^{1,\alpha}$ on $S_{1/2}$ which yields the desired estimate. In the following, we will only prove (\ref{est.u-l}) with $B_\rho = B_\rho(0)$, for the other cases are similar.
	
	Step 1: Perturbation. Let $g_0 = g(0)$.
	Let $A_0$ be a piecewise tensor defined by
	\begin{equation}\label{def.A0}
	A_0(x) = \left\{ 
	\begin{aligned}
	&A_+(0), \qquad \txt{if } x \in {B_{1,+}}, \\
	&A_-(0), \qquad \txt{if } x \in {B_{1,-}}.\\
	\end{aligned}
	\right.
	\end{equation}
	By subtracting constants, we may assume
	\begin{equation}\label{def.h0}
	h_+(0) = h_-(0) = 0.
	\end{equation}
	We then construct a family of functions that approximate $u$ at all scales. Precisely, let $v_t$ be the solution of
	\begin{equation}\label{eq.TP.A0}
	\left\{
	\begin{aligned}
	\nabla\cdot (A_0 \nabla v_t) &= 0 \qquad &\text{ in }& B_{t} \setminus S_t, \\	
	v_t & = u \qquad & \txt{ on } & \partial B_{t},\\ 
	(v_t)_+ &= (v_t)_- \qquad &\text{ on }&  S_t,\\
	\Big( \frac{\partial v_t}{\partial \nu_0} \Big)_+ - \Big( \frac{\partial v_t}{\partial \nu_0} \Big)_- &= g_0 \qquad &\text{ on }& S_t.
	\end{aligned} 
	\right.
	\end{equation}
	Because $S_t$ is flat and $g_0$ is constant, by Lemma \ref{lem.g0} and Lemma \ref{lem.C1a.excess}, for every $\rho \in (0,t)$ 
	\begin{equation}\label{est.vt.Lip}
	\bigg( \fint_{B_{\rho}} |\nabla v_t|^2 \bigg)^{1/2} \le C\bigg\{ \bigg( \fint_{B_{t}} |\nabla v_t|^2 \bigg)^{1/2} + |g_0|\bigg\},
	\end{equation}
	and
	\begin{equation}\label{est.vt.C1a}
	\begin{aligned}
	& \inf_{\ell \in \mathscr{L}} \bigg\{ \bigg( \fint_{B_{\rho}} |v_t - \ell|^2 \bigg)^{1/2} + \rho|g_0 - \mathscr{T}(\ell)| \bigg\} \\
	& \qquad \le C(\rho/t)^{1+\beta} \inf_{\ell \in \mathscr{L}} \bigg\{ \bigg( \fint_{B_{t}} |v_t - \ell|^2 \bigg)^{1/2} + t\Big|g_0 - \mathscr{T}(\ell) \Big|\bigg\}.
	\end{aligned}
	\end{equation}
	
	On the other hand, using the systems (\ref{eq.TP.C1a}), (\ref{eq.TP.A0}) and (\ref{def.var.sol}), for any $\phi \in H_0^1(B_t,\R^m)$, one has
	\begin{equation*}
	\int_{B_t} A(x)\nabla u\cdot \nabla \phi = \int_{B_t} h\cdot \nabla \phi + \int_{S_t} (g-n\cdot h_+ + n\cdot h_-)\cdot \phi d\sigma,
	\end{equation*}
	and
	\begin{equation*}
	\int_{B_t} A_0\nabla v_t \cdot \nabla \phi = \int_{S_t} g_0\cdot \phi d\sigma.
	\end{equation*}
	Combining these two identities, we obtain
	\begin{equation*}
	\begin{aligned}
	&\int_{B_t} A_0 \nabla(v_t - u)\cdot \nabla \phi \\
	& \qquad = \int_{B_t} (A-A_0)\nabla u\cdot \nabla\phi + \int_{B_t} (-h)\cdot \nabla \phi + \int_{S_t} (g_0-g + n\cdot (h_+ - h_-))\cdot \phi d\sigma.
	\end{aligned}
	\end{equation*}
	Taking $\phi = v_t - u \in H_0^1(B_t;\R^m)$, and using the trace theorem and a standard argument, we have
	\begin{equation}\label{est.Dvt-Du}
	\int_{B_t} |\nabla v_t - \nabla u|^2 \le Ct^{2\alpha} \bigg( \int_{B_t} |\nabla u|^2 + t^d \sum_{\pm} [h]_{C^\alpha(B_{t},\pm)}^2 + t^d [g]_{C^\alpha(S_t)}^2 \bigg),
	\end{equation}
	where we have used (\ref{def.h0}) to control the size of $h$.
	
	Step 2: Morrey estimate. Combining (\ref{est.vt.Lip}) and (\ref{est.Dvt-Du}), we have
	\begin{equation}\label{est.Du.Brho}
	\begin{aligned}
	\int_{B_\rho} |\nabla u|^2 &\le C(\rho/t)^d \bigg( \int_{B_t} |\nabla u|^2 + t^d|g_0|^2 \bigg) \\
	&\qquad + C t^{2\alpha} \bigg( \int_{B_t} |\nabla u|^2 + t^d \sum_{\pm} [h]_{C^\alpha(B_{t},\pm)}^2 + t^d [g]_{C^\alpha(S_t)}^2 \bigg).
	\end{aligned}
	\end{equation}
	Then, the Morrey estimate follows from a standard argument. Define
	\begin{equation}\label{def.Psi}
	\Psi(r) = \int_{B_r} |\nabla u|^2 + r^d \sum_{\pm} [h]_{C^\alpha(B_{r},\pm)}^2 + r^d [g]_{C^\alpha(S_r)}^2 + r^d|g_0|^2.
	\end{equation}
	Note that $\Psi(r)$ is a non-decreasing function. Then, (\ref{est.Du.Brho}) implies
	\begin{equation}\label{est.Psit}
	\Psi(\rho) \le C(\rho/t)^d\Psi(t) + Ct^{2\alpha} \Psi(t),
	\end{equation}
	for any $0<\rho < t<1/2$. Then by Lemma \ref{lem.Iteraton}, we have
	\begin{equation*}
	\Psi(\rho) \le C\rho^{2\alpha} \Psi(1/2), \qquad \txt{for any } \rho \in (0,1/2).
	\end{equation*}
	Note that this can be used to replace the last $\Psi(t)$ in (\ref{est.Psit}). It turns out that a bootstrap argument yields, for every $\gamma<d$,
	\begin{equation}\label{est.Psi.rho}
	\Psi(\rho) \le C\rho^\gamma \Psi(1/2).
	\end{equation}
	We point out again that even though the balls $B_t$'s are centered at the origin in the above proof, the same estimate actually holds if $B_t$'s are centered at any points in $B_{1/2}$. This implies that $|\nabla u|$ is in the Morrey space $L^{2,\gamma}(B_{1/2})$, which implies that $u$ is $C^\beta$-H\"{o}lder continuous for every $\beta\in (0,1)$. We will use the Morrey's estimate (\ref{est.Psi.rho}) in the next step.
	
	Step 3: $C^{1,\alpha}$ estimate. First, we combine (\ref{est.vt.C1a}), (\ref{est.Dvt-Du}) and the Poincar\'{e} inequality to obtain
	\begin{equation}\label{est.u-l.C1a}
	\begin{aligned}
	& \inf_{\ell \in \mathscr{L}} \bigg\{ \bigg( \int_{B_{\rho}} |u - \ell|^2 \bigg) + \rho^{d+2} \big|g_0 - \mathscr{T}(\ell)\big|^2 \bigg\} \\
	& \qquad \le C(\rho/t)^{d+2+2\beta} \inf_{\ell \in \mathscr{L}} \bigg\{ \bigg( \int_{B_{t}} |u - \ell|^2 \bigg) + t^{d+2}\big|g_0 - \mathscr{T}(\ell) \big|^2 \bigg\} \\
	& \qquad \qquad + t^{2+2\alpha} \bigg( \int_{B_t} |\nabla u|^2 + t^d \sum_{\pm} [h]_{C^\alpha(B_{t},\pm)}^2 + t^d [g]_{C^\alpha(S_t)}^2 \bigg).
	\end{aligned}
	\end{equation}
	Define
	\begin{equation*}
	\Phi(r) = \inf_{\ell \in \mathscr{L}} \bigg\{ \bigg( \int_{B_{r}} |u - \ell|^2 \bigg) + r^{d+2} |g_0 - \mathscr{T}(\ell)|^2 \bigg\},
	\end{equation*}
	and
	\begin{equation*}
	B(r) = \int_{B_t} |\nabla u|^2 + t^d \sum_{\pm} [h]_{C^\alpha(B_{t},\pm)}^2 + t^d [g]_{C^\alpha(S_t)}^2.
	\end{equation*}
	Now, using (\ref{est.Psi.rho}) with $\gamma = d-\alpha$, we have
	\begin{equation*}
	B(t) \le t^{d-\alpha} \Psi(1/2).
	\end{equation*}
	Then, it follows from (\ref{est.u-l.C1a}) that
	\begin{equation*}
	\Phi(\rho) \le C(\rho/t)^{d+2+2\beta} \Phi(t) + Ct^{d+2+\alpha} \Psi(1/2).
	\end{equation*}
	In view of Lemma \ref{lem.Iteraton}, by choosing $\beta\in (\alpha,1)$, we have
	\begin{equation*}
	\Phi(\rho) \le C\rho^{d+2+\alpha} (\Phi(1/2) + \Psi(1/2)),
	\end{equation*}
	for any $\rho \in (0,1/2)$. This implies that $u \in C^{1,\alpha/2}(\overline{ B_{1/2,\pm}};\R^m)$. To increase the exponent from $\alpha/2$ to $\alpha$, we have to bootstrap the argument one more time. In particular, the $C^{0,1}$ estimate implies
	\begin{equation*}
	\int_{B_t} |\nabla u|^2 \le C t^d (\Phi(1/2) + \Psi(1/2)).
	\end{equation*}
	Inserting this estimate into (\ref{est.u-l.C1a}), we obtain
	\begin{equation*}
	\Phi(\rho) \le C(\rho/t)^{d+2+2\beta} \Phi(t) + Ct^{d+2+2\alpha} (\Phi(1/2) + \Psi(1/2)).
	\end{equation*}
	Because $\beta>\alpha$, we may apply Lemma \ref{lem.Iteraton} one more time to obtain
	\begin{equation*}
	\Phi(\rho) \le C\rho^{d+2+2\alpha} (\Phi(1/2) + \Psi(1/2)),
	\end{equation*}
	which leads to the desired estimate (\ref{est.u-l}).
\end{proof}

\subsection{General $C^{k,\alpha}$ estimates}
With the $C^{1,\alpha}$ estimate at our disposal, the proof of Theorem \ref{thm.Cka} is almost routine.

\begin{proof}[Proof of Theorem \ref{thm.Cka}]
	By localization and flattening the interface, it suffices to prove the theorem for the case $S_1 = B_1 \cap \{x_d = 0\}$ as in Theorem \ref{lem.const.C1a}, i.e.,
	\begin{equation}\label{eq.TP.var}
	\left\{
	\begin{aligned}
	\nabla\cdot (A(x) \nabla u) &= \nabla\cdot h \qquad &\text{ in }& B_1 \setminus S_1 \\
	u_+ &= u_- \qquad &\text{ on }&  S_1\\
	\Big( \frac{\partial u}{\partial \nu} \Big)_+ - \Big( \frac{\partial u}{\partial \nu} \Big)_- &= g \qquad &\text{ on }& S_1,
	\end{aligned} 
	\right.
	\end{equation}
	where
	\begin{equation*}
	\Big( \frac{\partial u}{\partial \nu} \Big)_{\pm}(P) = (-e_d)\cdot A_{\pm}(P) (\nabla u)_{\pm}(P),
	\end{equation*}
	for $P\in S_1$. Theorem \ref{lem.const.C1a} implies the case $k=1$ with any $\alpha\in (0,1)$. To prove the case $k=2$, we take derivatives to the system (\ref{eq.TP.var}) for the first $d-1$ variables. Precisely, let $1\le j\le d-1$, and set $w = \partial u/\partial x_j$. Then $w$ satisfies
	\begin{equation}
	\left\{
	\begin{aligned}
	\nabla\cdot (A(x) \nabla w) &= -\nabla\cdot \Big( \frac{\partial}{\partial x_j} A(x) \nabla u \Big) + \nabla\cdot \Big( \frac{\partial h}{\partial x_j} \Big) \quad &\text{ in }& B_1 \setminus S_1 \\
	w_+ &= w_- \quad &\text{ on }&  S_1\\
	\Big( \frac{\partial w}{\partial \nu} \Big)_+ - \Big( \frac{\partial w}{\partial \nu} \Big)_- &= \frac{\partial}{\partial x_j} g + e_d\cdot \Big( \frac{\partial A}{\partial x_j}\Big)_+ (\nabla u)_+ - e_d\cdot \Big( \frac{\partial A}{\partial x_j}\Big)_- (\nabla u)_- \quad &\text{ on }& S_1.
	\end{aligned} 
	\right.
	\end{equation}

	Applying Theorem \ref{lem.const.C1a}, we see that $w\in C^{1,\alpha}(\overline{B_{1/2,\pm}};\R^m)$ and
	\begin{equation}\label{est.k=2xj}
	\begin{aligned}
	& \sum_{\pm} \Norm {\frac{\partial u}{\partial x_j}}_{C^{1,\alpha}(\overline{B_{1/2,\pm}})} = \sum_{\pm} \norm {w}_{C^{1,\alpha}(\overline{B_{1/2,\pm}})}\\
	& \le C\big(  \sum_{\pm} \norm {\nabla A \nabla u}_{C^{\alpha}(\overline{B_{3/4,\pm}})} + \sum_{\pm} \norm {\nabla h}_{C^{\alpha}(\overline{B_{1,\pm}})} + \norm {\nabla_{\txt{tan}}g }_{C^\alpha(S_1)} + \norm{w}_{L^2(B_{3/4})} \big) \\
	& \le C\big(  \sum_{\pm} \norm {h}_{C^{1,\alpha}(\overline{B_{1,\pm}})} + \norm {g }_{C^{1,\alpha}(S)} + \norm {u}_{L^2(B_1)}\big),
	\end{aligned}
	\end{equation}
	for any $1\le j\le d-1$. To recover the last derivative in the direction of $x_d$, we use the system (\ref{eq.TP.var}) to rearrange the equations as (\ref{eq.Dd2u}) with $C^\alpha$ coefficients. In view of (\ref{est.k=2xj}), this gives the desired estimate for $\partial^2 u/\partial x_d^2$ and
	\begin{equation*}
	\sum_{\pm} \norm{u}_{C^{2,\alpha}(\overline{B_{1/2,\pm}})} \le C\big( \sum_{\pm} \norm {h}_{C^{1,\alpha}(\overline{B_{1,\pm}})} + \norm {g }_{C^{1,\alpha}(S_1)} + \norm {u}_{L^2(B_1)}\big).
	\end{equation*}
	This proves the case $k=2$. Now the general $k>2$ may be proved inductively by mimicking the above process.
\end{proof}

The following corollary particularly recovers the sharp Schauder estimates for the system (\ref{eq.Disc.Source}) in \cite{Dong12, XB13} in the case of two components.

\begin{corollary}
	Let $S, A_\pm$ and $h_\pm$ satisfy the same assumptions as in Theorem \ref{thm.Cka}. Suppose $u$ is a weak solution of
	\begin{equation}\label{eq.B1.Dh}
	\nabla\cdot (A(x) \nabla u) = \nabla\cdot h  \qquad \txt{ in } B_1.
	\end{equation}
	Then, $u\in C^{k,\alpha}(\overline{B_{1/2,\pm}};\R^m)$ and
	\begin{equation*}
	\sum_{\pm} \norm{u}_{C^{k,\alpha}(\overline{B_{1/2,\pm}})} \le C \Big( \sum_{\pm} \norm {h}_{C^{k-1,\alpha}({B_{1,\pm}})} + \norm {u}_{L^2(B_1)} \Big),
	\end{equation*}
	where $C$ depends only on $\Lambda,d,m,k,\alpha$, $C^{k,\alpha}$ character of $S$ and $C^{k-1,\alpha}$ character of $A_{\pm}$.
\end{corollary}
\begin{proof}
	This follows from Theorem \ref{thm.Cka} and the fact that (\ref{eq.B1.Dh}) may be equivalently interpreted as (see Remark \ref{rmk.specialcase})
	\begin{equation}
	\left\{
	\begin{aligned}
	\nabla\cdot (A(x) \nabla u) &= \nabla\cdot h \qquad &\text{ in }& B_1 \setminus S, \\
	u_+ &= u_- \qquad &\text{ on }&  S,\\
	\Big( \frac{\partial u}{\partial \nu} \Big)_+ - \Big( \frac{\partial u}{\partial \nu} \Big)_- &= n\cdot h_+ - n \cdot h_-  \qquad &\text{ on }& S.
	\end{aligned} 
	\right.
	\end{equation}
	Clearly, the above system is a special case of (\ref{eq.TP}). Because $S$ is $C^{1,\alpha}$, then $n$ is $C^\alpha$ on $S$ and so is $n\cdot (h_+ - h_-)$. Thus the desired estimate follows readily from Theorem \ref{thm.Cka}.
\end{proof}

\section{Homogenization}
In this section, we study the transmission problem with periodically oscillating coefficients and establish algebraic rates of convergence.

We first recall the classical homogenization theory with periodic coefficients; see \cite{Sh18}. Let $Y$ be a parallelotope in $\R^d$ centered at $0\in \R^d$ and $A$ be a $Y$-periodic coefficient tensor. The correctors $\chi_j^\beta = (\chi_j^{\alpha\beta})$, with $1\le j\le d$ and $1\le \beta \le m$, are $Y$-periodic functions in $H^1(Y;\R^m)$ solving the following cell problem
\begin{equation}\label{def.corrector}
\left\{
\begin{aligned}
\nabla\cdot (A \nabla \chi_j^\beta) &= -\nabla\cdot (Ae_j^\beta) \quad &\text{ in }& Y,\\
\int_Y \chi_j^\beta &= 0, &&
\end{aligned} 
\right.
\end{equation}
where $e_j^\beta = e_j\otimes e^\beta$, $e_j = (0,\cdots,1,\cdots,0)$ is the vector in $\R^d$ with $1$ in the $j$th position and $e^\beta = (0,\cdots,1,\cdots,0)$ is the vector in $\R^m$ with $1$ in the $\beta$th position. The homogenized coefficient tensor of $A$ is given by $\widehat{A} = (\widehat{a}_{ij}^{\alpha\beta})$ and
\begin{equation}\label{def.homogenized.coeff}
\widehat{a}_{ij}^{\alpha\beta} = \fint_Y \bigg[ a_{ij}^{\alpha\beta} + a_{ik}^{\alpha\gamma} \frac{\partial}{\partial x_k} \chi_j^{\gamma\beta} \bigg].
\end{equation}

Let
\begin{equation}\label{def.bcorrector}
b_{ij}^{\alpha\beta} = a_{ij}^{\alpha\beta} + a_{ik}^{\alpha\gamma} \frac{\partial}{\partial x_k} \chi_j^{\gamma\beta} - \widehat{a}_{ij}^{\alpha\beta}.
\end{equation}
The system (\ref{def.corrector}) and (\ref{def.homogenized.coeff}) implies
\begin{equation*}
\frac{\partial}{\partial y_i} b_{ij}^{\alpha\beta} = 0 \qquad \txt{and} \qquad \int_Y b_{ij}^{\alpha\beta} = 0.
\end{equation*}
As in \cite[Chapter 2.1]{Sh18}, this implies that there exists a sequence of $Y$-periodic functions $\phi_{kij}^{\alpha\beta}$, with $1\le i,j,k\le d$ and $1\le \alpha,\beta \le m$, such that
\begin{equation}\label{eq.flux}
b_{ij}^{\alpha\beta} = \frac{\partial}{\partial y_k} \big( \phi_{kij}^{\alpha\beta} \big) \qquad \txt{and} \qquad \phi_{kij}^{\alpha\beta} = -\phi_{ikj}^{\alpha\beta}.
\end{equation}
The functions $(\phi_{kij}^{\alpha\beta})$ are usually called the flux correctors.

Let $\Omega$ be a bounded Lipschitz domain and $\Omega_{\pm}$ be two disjoint Lipschitz subdomains such that $\Omega_{-} = \Omega \setminus \overline{\Omega_{+}}$. Let $S = \overline{\Omega_{+}} \cap \overline{\Omega_{-}}$ be the interface. Let $\e_+, \e_- \in (0,\infty)$ be two scale parameters that might be distinct, and $\e = \{\e_+,\e_-\}$. Define
\begin{equation}\label{def.Ae}
A^\e(x) = \left\{ 
\begin{aligned}
A_+^{\e_+}(x) &= A_+(x/\e_+), \qquad \txt{if } x \in \overline{\Omega_+} \\
A_-^{\e_-}(x) &= A_-(x/{\e_-}), \qquad \txt{if } x \in \Omega_-.\\
\end{aligned}
\right.
\end{equation}
We assume that $A_{\pm}$ are $Y_{\pm}$-periodic coefficient tensors, where $Y_+$ and $Y_-$ are completely independent parallelotopes. We denote by $\widehat{A}_{\pm}$, $\chi_{\pm,j}^{\alpha\beta}$ and $\phi_{\pm,kij}^{\alpha\beta}$ the homogenized operator, correctors and flux correctors associated with $A_{\pm}$, respectively. Also, for any $P\in S$, define the conormal derivatives by
\begin{equation*}
\Big( \frac{\partial u}{\partial \nu_\e} \Big)_{\pm}(P) = n\cdot A^{\e_\pm}_{\pm}(P) (\nabla u)_{\pm}(P).
\end{equation*}

Consider a transmission problem with periodically oscillating coefficients
\begin{equation}\label{eq.TP.periodic}
\left\{
\begin{aligned}
\nabla\cdot (A^\e \nabla u^\e) &= F \qquad &\text{ in }& \Omega \setminus S, \\
u^\e & = f \qquad &\txt{on } & \partial \Omega, \\
u^\e_+ &= u^\e_- \qquad &\text{ on }&  S,\\
\Big( \frac{\partial u^\e}{\partial \nu_\e} \Big)_+ - \Big( \frac{\partial u^\e}{\partial \nu_\e} \Big)_- &= g \qquad &\text{ on }& S,
\end{aligned} 
\right.
\end{equation}
where we assume $F\in L^2(\Omega;\R^m), f\in H^{1/2}(\partial\Omega;\R^m)$, $g\in H^{-1/2}(S;\R^m)$.

The system (\ref{eq.TP.periodic}) is a multi-scale problems (with 3 scales, $\e_+, \e_-$ and $\txt{diam}(\Omega)$). In the following subsections, we will show quantitatively that if both $\e_\pm \to 0$, then (\ref{eq.TP.periodic}) homogenizes to the unique homogenized system in both components $\Omega_{\pm}$; if $\e_+\to 0$ and $\e_- \to \infty$, then (\ref{eq.TP.periodic}) homogenizes only in $\Omega_{+}$ and exhibits a blow-up behavior in $\Omega_-$, provided $A_-$ is H\"{o}lder continuous. The latter case will take place if a problem of mesoscopic scale is rescaled.


\subsection{Case $0<\e_\pm<1$}

In the following theorem, we show that quantitatively, (\ref{eq.TP.periodic}) homogenizes to
\begin{equation}\label{eq.TP.homogenized}
\left\{
\begin{aligned}
\nabla\cdot (\widehat{A} \nabla u^0) &= F \qquad &\text{ in }& \Omega \setminus S, \\
u^0 & = f \qquad &\txt{on } & \partial \Omega, \\
u^0_+ &= u^0_- \qquad &\text{ on }&  S,\\
\Big( \frac{\partial u^0}{\partial \nu_0} \Big)_+ - \Big( \frac{\partial u^0}{\partial \nu_0} \Big)_- &= g \qquad &\text{ on }& S,
\end{aligned} 
\right.
\end{equation}
where the homogenized coefficient tensor $\widehat{A}$ is piecewise constant, namely,
\begin{equation}\label{def.Ahat}
\widehat{A} = \left\{ 
\begin{aligned}
&\widehat{A}_+, \qquad \txt{if } x \in \overline{\Omega_+}, \\
&\widehat{A}_-, \qquad \txt{if } x \in \Omega_-.\\
\end{aligned}
\right.
\end{equation}
Also, the corresponding conormal derivatives are given by
\begin{equation*}
\Big( \frac{\partial u^0 }{\partial \nu_0} \Big)_{\pm}(P) = n\cdot \widehat{A}_{\pm} (\nabla u^0)_{\pm}(P).
\end{equation*}

\begin{theorem}\label{thm.TP.rate}
	Let $\e = \{\e_+,\e_-\}$ and $\e_{\pm} \in (0,1)$. There exists $\sigma>0$ depending only on $\Lambda,p,\alpha$ and the Lipschitz characters of $\partial\Omega$ and $S$ such that
	\begin{equation*}
	\norm {u^\e - u^0}_{L^2(\Omega)} \le C\sum_{\pm} \e^\sigma_\pm \Big( \norm {F}_{L^2(\Omega)} + \norm {f}_{H^{1/2}(\partial\Omega)}  + \norm {g}_{H^{-1/2}(S)} \Big),
	\end{equation*}
	where $u^\e$ and $u^0$ solve (\ref{eq.TP.periodic}) and (\ref{eq.TP.homogenized}), respectively.
\end{theorem}

Let $t_\pm>0$ and $\Omega_{\pm}(t_\pm) = \{ x\in \Omega_{\pm}: \dist(x,\partial \Omega_{\pm}) < t_\pm \}$. Let $\eta_{\pm}\in C_0^\infty(\Omega_{\pm})$ be a cut-off function such that $\eta_{\pm} = 0$ on $\Omega_{\pm}(2t_\pm), \eta_\pm = 1$ on $\Omega_{\pm} \setminus \Omega_{\pm}(4t_\pm)$ and $|\nabla \eta_\pm| \le Ct^{-1}_\pm$ in $\Omega_{\pm}$. Since we assume no smoothness on the coefficients $A_\pm$, we need the following smoothing operator:
\begin{equation}\label{def.St}
S_t f(x) = t^{-d}\int_{\R^d} f(x-y)  \varphi(y/t) dy,
\end{equation}
where $\varphi\in C_0^\infty(B_1)$ is a non-negative function with $\int_{\R^d} \varphi = 1$. Note that $S_t$ commutes with $\partial/\partial x_j$. Many useful properties of $S_t$ may be found in \cite{SZ17,Sh17} or \cite[Chapter 2.1]{Sh18}.
To avoid the distracting supscripts, we will only prove the theorem for scalar case (i.e., $m=1$). But the argument for general system is exactly the same. 

The key insight of the proof is to estimate the error of the first-order approximation constructed below
\begin{equation}\label{def.we}
w^\e = u^\e - u^0 - \e_+ \chi_{+,j}^{\e_+} S_{\e_+} \bigg( \frac{\partial u^0}{\partial x_j} \bigg) \eta_{+} - \e_- \chi_{-,j}^{\e_-} S_{\e_-} \bigg( \frac{\partial u^0}{\partial x_j} \bigg) \eta_{-}.
\end{equation}
Note that we need two first-order correction terms since $A_+ \neq A_-$. As usual, one may expect to carry out a familiar argument to estimate the $H^1$ norm of $w^\e$. However, because of the lack of the a priori regularity estimate of $u^0$ (such as the Meyers' estimate) under our weak assumptions (for example, the transmission data $g$ belongs merely in $H^{-1/2}(S;\R^m)$, instead of $L^2(S;\R^m)$), we are unable to show the convergence rate in $H^1$ directly. In the following, we will use a duality technique, developed in \cite{ZP05,Gr06,Sus13,SZ17}, to overcome this difficulty. The idea is to take advantage of the a priori regularity of the dual problem, which is well-known. This trick should be very powerful in dealing with problems with rough interfaces or rough transmission data.

The next lemma is in the spirit of \cite[Lemma 3.5]{SZ17}
\begin{lemma}\label{lem.dual}
	Let $t_\pm \in (2\e_{\pm},1)$ be arbitrary and  $w^\e$ be as in (\ref{def.we}). Then for any $\psi \in H^1_0(\Omega;\R^m)$,
	\begin{equation}\label{est.DweDpsi}
	\begin{aligned}
	\bigg| \int_{\Omega} A^\e \nabla w^\e \cdot \nabla \psi \bigg|
	& \le C\Big( \norm {F}_{L^2(\Omega)} + \norm {f}_{H^{1/2}(\partial\Omega)} + \norm {g}_{H^{-1/2}(\Omega)} \Big) \\
	& \qquad \times \Big( \sum_{\pm} \e_\pm t^{-1}_\pm \norm {\nabla \psi}_{L^2(\Omega)} + \sum_{\pm} \norm {\nabla \psi}_{L^2(\Omega_\pm(4t_\pm))} \Big).
	\end{aligned}
	\end{equation}
\end{lemma}

\begin{proof}
	A direct computation combining (\ref{eq.TP.periodic}) and (\ref{eq.TP.homogenized}) leads to
	\begin{align*}
	\nabla\cdot (A^\e \nabla w^\e) & = \nabla\cdot \bigg\{ \sum_{\pm} \Big[(\widehat{A}_{\pm} - A^{\e_\pm}_\pm - A^{\e_\pm}_\pm \nabla \chi_{\pm}^{\e_\pm}) S_{\e_\pm}(\nabla u^0) \eta_{\pm} \Big] \bigg\} \\
	& \qquad - \nabla\cdot \Big[ \sum_{\pm} (\widehat{A}_\pm - A^{\e_\pm}_\pm ) \big(\nabla u^0 \mathbbm{1}_{\Omega_{\pm}} -S_{\e_\pm}(\nabla u^0)\eta_{\pm}\big)   \Big] \\
	& \qquad - \nabla \cdot \bigg[ \sum_{\pm} \e_\pm A^{\e_\pm} \chi_{\pm,j}^{\e_\pm} \nabla \bigg( S_{\e_\pm} \bigg( \frac{\partial u^0}{\partial x_j} \bigg) \eta_{\pm} \bigg)  \bigg] \\
	& =: \sum_{i = 1}^{3} \sum_{\pm} \nabla\cdot J_i^\pm.
	\end{align*}
	Now, to simplify $\nabla\cdot J_1^\pm$, we need to use the flux correctors. For $J_1^+$, by (\ref{def.bcorrector}) and (\ref{eq.flux}),
	\begin{equation*}
	\widehat{A}_+ - A^{\e_+}_+ - A^{\e_+}_+ \nabla \chi_{+}^{\e_+} = -\frac{\partial}{\partial x_k} \big( \e_+ \phi_{+,k}^{\e_+} \big).
	\end{equation*}
	Hence,
	\begin{align*}
	\nabla\cdot J_1^+ &= \frac{\partial}{\partial x_i} \bigg[ -\frac{\partial}{\partial x_k}\big( \e_+ \phi_{+,kij}^{\e_+} \big) S_{\e_+}\bigg( \frac{\partial u^0}{\partial x_j} \bigg) \eta_{+} \bigg] \\
	& = \frac{\partial}{\partial x_i} \bigg[  -\e_+ \phi_{+,kij}^{\e_+} \frac{\partial}{\partial x_k} \bigg( S_{\e_+}\bigg( \frac{\partial u^0}{\partial x_j} \bigg) \eta_{+} \bigg) \bigg] =: \nabla\cdot \tilde{J}_1^{+},
	\end{align*}
	where we have used the skew-symmetry of $\phi_{kij}$ with respect to $i$ and $k$ in the second equality. Similarly, we may write $\nabla\cdot J_1^- = \nabla\cdot \tilde{J}_1^-$, where $\tilde{J}_1^-$ could be self-explained according to $\tilde{J}_1^+$.
	Hence, it follows from the integration by parts that
	\begin{equation}\label{eq.DweDpsi}
	\int_{\Omega} A^\e \nabla w^\e \cdot \nabla \psi = -\sum_{\pm} \int_{\Omega} \tilde{J}_1^\pm \cdot \nabla \psi - \sum_{i = 2}^3 \sum_{\pm} \int_{\Omega} J_i^\pm \cdot \nabla \psi.
	\end{equation}
	
	Now, by the product rule and the triangle inequality, we have
	\begin{equation}\label{est.J1+}
	\begin{aligned}
	\norm {\tilde{J}_{1}^{+}}_{L^2(\Omega)} & \le \e_+  \norm {\phi^{\e_+} S_{\e_+}(\nabla^2 u^0)}_{L^2(\Omega_+\setminus \Omega_+(2t_+))} \\
	& \qquad + C\e_+ t^{-1}_+ \norm {\phi^{\e_+} S_{\e_+}(\nabla u^0)}_{L^2(\Omega_+(4t_+)\setminus \Omega_+(2t_+) )} .
	\end{aligned}
	\end{equation}
	Observe that $\nabla\cdot (\widehat{A}_+ \nabla u^0) = F$ in $\Omega_+$ and $\widehat{A}_+$ is constant in $\Omega_+$. The classical interior $H^2$ regularity for elliptic system and \cite[Lemma 2.1]{Sh17} implies that
	\begin{equation}\label{est.D2u0}
	\begin{aligned}
	\norm {\phi^{\e_+} S_{\e_+}(\nabla^2 u^0)}_{L^2(\Omega_+\setminus \Omega_+(2t_+))} & \le C \norm {\nabla^2 u^0}_{L^2(\Omega_+\setminus \Omega_+(t_+))} \\
	&\le Ct^{-1}_+ \norm {\nabla u^0}_{L^2(\Omega_+)} + C\norm {F}_{L^2(\Omega_+)}.
	\end{aligned}
	\end{equation}
	On the other hand, using \cite[Lemma 2.1]{Sh17} again
	\begin{equation}\label{est.Du0.layer}
	\norm {\phi^{\e_+} S_{\e_+}(\nabla u^0)}_{L^2(\Omega_+(4t_+)\setminus \Omega_+(2t_+) )} \le C\norm {\nabla u^0}_{L^2(\Omega_+(5t_+))}.
	\end{equation}
	Combining (\ref{est.J1+}) - (\ref{est.Du0.layer}), we have
	\begin{equation*}
	\norm {\tilde{J}_1^+}_{L^2(\Omega)} \le C\e_+ t^{-1}_+ \norm {\nabla u^0}_{L^2(\Omega_+)} + C \e\norm {F}_{L^2(\Omega_+)}.
	\end{equation*}
	Consequently,
	\begin{equation}\label{est.J'1}
	\bigg| \int_{\Omega} \tilde{J}_1^+ \cdot \nabla \psi \bigg| \le C\e_+ t^{-1}_+ \norm {\nabla \psi}_{L^2(\Omega)} \Big( \norm {F}_{L^2(\Omega)} + \norm {f}_{H^{1/2}(\partial\Omega)} + \norm {g}_{H^{-1/2}(\Omega)} \Big).
	\end{equation}
	This gives the desired estimate for the first integral on the right-hand side of (\ref{eq.DweDpsi}). Note that the estimates of $\int_{\Omega} \tilde{J}_1^-\cdot \nabla \psi$ and $\int_{\Omega} J_3^\pm \cdot \nabla \psi$ are the same as (\ref{est.J'1}).
	
	It suffices to estimate $\int_{\Omega} J_2^\pm\cdot \nabla \psi$. We only consider $\int_{\Omega} J_2^+ \cdot \nabla \psi$. By the triangle inequality,
	\begin{equation*}
	\begin{aligned}
	& \bigg|\int_{\Omega} J_2^+\cdot \nabla \psi  \bigg| \\
	& \le C\int_{\Omega_+} |(\nabla u^0 - S_{\e_+}(\nabla u^0)) \eta_{+} | |\nabla \psi| + C\int_{\Omega_+} |\nabla u^0 (1-\eta_{+})| |\nabla \psi| \\
	& \le C\norm {\nabla \psi}_{L^2(\Omega)} \Big( \e_+ t^{-1}_+ \norm {\nabla u^0}_{L^2(\Omega)} + \e_+ \norm {F}_{L^2(\Omega)} \Big) + C\norm {\nabla \psi}_{L^2(\Omega_{+}(4t_+)  )} \norm {\nabla u^0}_{L^2(\Omega)},
	\end{aligned}
	\end{equation*}
	where we have used \cite[Lemma 2.2]{Sh17} and (\ref{est.D2u0}) in estimating the first integral and used the fact that $1-\eta_{+}$ is supported in $\Omega_{+}(4t)$ for the second integral. Combining all the estimates for the right-hand side of (\ref{eq.DweDpsi}), we obtain (\ref{est.DweDpsi}).
\end{proof}

\begin{proof}[Proof of Theorem \ref{thm.TP.rate}]
	For any given $G\in L^2(\Omega;\R^m)$, let $v^\e$ be the weak solution of 
	\begin{equation}\label{eq.dual}
	\left\{
	\begin{aligned}
	\nabla\cdot (A^{\e*} \nabla v^\e) &= G \qquad &\text{ in }& \Omega, \\
	v^\e & = 0 \qquad &\txt{on } & \partial \Omega.
	\end{aligned} 
	\right.
	\end{equation}
	Since this is a regular elliptic system, the Meyers' estimate holds, namely, there is some $p>2$, depending only on $\Lambda$ and the Lipschitz character of $\partial\Omega$, such that
	\begin{equation}\label{est.ve.Meyer}
	\norm {\nabla v^\e}_{L^p(\Omega)} \le C \norm {G}_{L^2(\Omega)}.
	\end{equation}
	Let $u^\e$ and $u^0$ solve (\ref{eq.TP.periodic}) and (\ref{eq.TP.homogenized}), respectively. Let $w^\e$ be given by (\ref{def.we}). Then, by the integration by parts and Lemma \ref{lem.dual}, we have
	\begin{align}\label{est.we.G}
	\begin{aligned}
	\bigg| \int_{\Omega} w^\e \cdot G \bigg| & = \bigg| \int_{\Omega} A^\e \nabla w^\e \cdot  \nabla v^\e \bigg| \\
	& \le C\Big( \norm {F}_{L^2(\Omega)} + \norm {f}_{H^{1/2}(\partial\Omega)} + \norm {g}_{H^{-1/2}(\Omega)} \Big) \\
	& \qquad \times \Big( \sum_{\pm} \e_\pm t^{-1}_\pm \norm {\nabla v^\e }_{L^2(\Omega)} + \sum_{\pm} \norm {\nabla v^\e}_{L^2(\Omega_\pm(4t_\pm))} \Big).
	\end{aligned}
	\end{align}
	Now, (\ref{est.ve.Meyer}) implies
	\begin{equation*}
	\norm {\nabla v^\e}_{L^2(\Omega_\pm(4t))} \le C t^{1/2-1/p}_\pm \norm {G}_{L^2(\Omega)}.
	\end{equation*}
	Inserting this into (\ref{est.we.G}), we obtain
	\begin{equation*}
	\begin{aligned}
	\bigg| \int_{\Omega} w^\e \cdot G \bigg| & \le C \Big(\sum_{\pm} \e_\pm t^{-1}_\pm + t^{1/2-1/p}_\pm \Big) \norm {G}_{L^2(\Omega)} \Big( \norm {F}_{L^2(\Omega)} + \norm {f}_{H^{1/2}(\partial\Omega)} + \norm {g}_{H^{-1/2}(\Omega)} \Big) \\
	& \le C\sum_{\pm} \e_\pm^\sigma \norm{G}_{L^2(\Omega)} \Big( \norm {F}_{L^2(\Omega)} + \norm {f}_{H^{1/2}(\partial\Omega)} + \norm {g}_{H^{-1/2}(\Omega)} \Big),
	\end{aligned}
	\end{equation*}
	where, in the last inequality, we have chosen $t_\pm = \e^{\sigma_1}_\pm$ with $ \sigma_1 = (3/2-1/p)^{-1}<1$ and $\sigma = 1-\sigma_1>0$. This implies the desired estimate.
\end{proof}

\subsection{Case $0<\e_+ < 1 < \e_-$}
In this subsection, we show that, as $\e_+ \to 0$ and $\e_- \to \infty$, (\ref{eq.TP.periodic}) converges quantitatively to the following system
\begin{equation}\label{eq.TP.homogenized0}
\left\{
\begin{aligned}
\nabla\cdot (\overline{A} \nabla \bar{u}) &= F \qquad &\text{ in }& \Omega \setminus S, \\
\bar{u} & = f \qquad &\txt{on } & \partial \Omega, \\
\bar{u}_+ &= \bar{u}_- \qquad &\text{ on }&  S,\\
\Big( \frac{\partial \bar{u}}{\partial \nu_0} \Big)_+ - \Big( \frac{\partial \bar{u}}{\partial \nu^\circ} \Big)_- &= g \qquad &\text{ on }& S,
\end{aligned} 
\right.
\end{equation}
where
\begin{equation}\label{def.Abar}
\overline{A} = \left\{ 
\begin{aligned}
&\widehat{A}_+, \qquad &\txt{if }& x \in \overline{\Omega_+}, \\
&A_-^\circ :=A_-(0), \qquad &\txt{if }& x \in \Omega_-,\\
\end{aligned}
\right.
\end{equation}
and
\begin{equation*}
\Big( \frac{\partial \bar{u}}{\partial \nu_0} \Big)_+ = n\cdot \widehat{A}_+ \nabla \bar{u}, \quad \Big( \frac{\partial \bar{u}}{\partial \nu^\circ} \Big)_+ = n\cdot A_-^\circ \nabla \bar{u}.
\end{equation*}
Without loss of generality, we may assume $0\in \Omega$.

\begin{theorem}\label{thm.TP.rate0}
	Let $\e = \{\e_+,\e_-\}$ and $\e_{+} \in (0,1)$ and $\e_- \in (1,\infty)$. In addition, we assume $A_-$ is $C^\alpha$-H\"{o}lder continuous. Then there exists $\sigma>0$ depending only on $\Lambda,p,\alpha$ and the Lipschitz character of $\partial\Omega$ such that
	\begin{equation*}
	\norm {u^\e - \bar{u}}_{L^2(\Omega)} \le C(\e^\sigma_+ + \e^{-\alpha}_-)\Big( \norm {F}_{L^2(\Omega)} + \norm {f}_{H^{1/2}(\partial\Omega)}  + \norm {g}_{H^{-1/2}(S)} \Big),
	\end{equation*}
	where $u^\e$ and $\bar{u}$ solve (\ref{eq.TP.periodic}) and (\ref{eq.TP.homogenized0}), respectively.
\end{theorem}

The proof is similar to Theorem \ref{thm.TP.rate} by a duality argument. We begin with a lemma similar to (\ref{lem.dual}).

\begin{lemma}
	Let the assumptions in Theorem \ref{thm.TP.rate0} hold. Let $t_+ \in (2\e_+,1)$ be arbitrary and
	\begin{equation*}
	w^\e = u^\e - \bar{u} - \e_+ \chi_{+,j}^{\e_+} S_{\e_+} \bigg( \frac{\partial \bar{u}}{\partial x_j} \bigg) \eta_{+}.
	\end{equation*}
	Then for any $\psi \in H^1_0(\Omega;\R^m)$,
	\begin{equation}
	\begin{aligned}
	\bigg| \int_{\Omega} A^\e \nabla w^\e \cdot \nabla \psi \bigg|
	& \le C\Big( \norm {F}_{L^2(\Omega)} + \norm {f}_{H^{1/2}(\partial\Omega)} + \norm {g}_{H^{-1/2}(\Omega)} \Big) \\
	& \qquad \times \Big( (\e_+ t^{-1}_+ + \e_-^{-\alpha}) \norm {\nabla \psi}_{L^2(\Omega)} +  \norm {\nabla \psi}_{L^2(\Omega_+(4t_+))} \Big).
	\end{aligned}
	\end{equation}
\end{lemma}
\begin{proof}
	A direct computation combining (\ref{eq.TP.periodic}) and (\ref{eq.TP.homogenized0}) shows that
	\begin{align*}
	\nabla\cdot (A^\e \nabla w^\e) & = \nabla\cdot \bigg\{ \bigg[(\widehat{A}_{+} - A^{\e_+}_+ - A^{\e_+}_+ \nabla \chi_{+,j}^{\e_+}) S_{\e_+}\bigg( \frac{\partial \bar{u}}{\partial x_j} \bigg) \eta_{+} \bigg] \bigg\} \\
	& \qquad - \nabla\cdot \Big[ (\widehat{A}_+ - A^{\e_+}_+ ) \big(\nabla u^0 \mathbbm{1}_{\Omega_{+}} -S_{\e_+}(\nabla \bar{u})\eta_{+}\big)   \Big] \\
	& \qquad - \nabla \cdot \bigg[ \e_+ A^{\e_+} \chi_{+,j}^{\e_+} \nabla \bigg( S_{\e_+} \bigg( \frac{\partial \bar{u}}{\partial x_j} \bigg) \eta_{+} \bigg)  \bigg] \\
	& \qquad + \nabla \cdot (\mathbbm{1}_{\Omega_{-}} (A_-^\circ - A_-^{\e_-}) \nabla \bar{u}  )\\
	& = \sum_{i = 1}^{4} \nabla\cdot I_i.
	\end{align*}
	By the integration by parts, for any $\phi \in H_0^1(\Omega;\R^m)$,
	\begin{equation*}
	\int_{\Omega} A^\e \nabla w^\e \cdot \nabla \psi = - \sum_{i = 1}^{4} \int_{\Omega} I_i\cdot \nabla \psi.
	\end{equation*}
	Now, the estimates for the first three integrals are the same as in Lemma \ref{lem.dual}. The last integral follows easily by noting that
	\begin{equation*}
	|A_-^\circ - A_-^{\e_-}| = |A(0) - A(x/\e_-)| \le |x/\e_-|^\alpha [A_-]_{C^\alpha} \le \e_-^{-\alpha} [A_-]_{C^\alpha} \txt{diam}(\Omega)^\alpha,
	\end{equation*}
	since $A_-$ is $C^\alpha$-H\"{o}lder continuous and $0\in \Omega$. As a result,
	\begin{equation*}
	\bigg| \int_{\Omega} I_4\cdot \nabla \psi \bigg| \le C\e_-^{-\alpha} \norm {\nabla \bar{u}}_{L^2(\Omega_{-})} \norm {\nabla \psi}_{L^2(\Omega_{-})}.
	\end{equation*}
	This completes the proof.
\end{proof}

\begin{proof}[Proof of Theorem \ref{thm.TP.rate0}]
	The proof is almost the same as the proof of Theorem \ref{thm.TP.rate} and is left to the reader.
\end{proof}

\subsection{Convergence rates}
For our application, we will prove the Lipschitz estimate in the case $F = \nabla\cdot h$ in (\ref{eq.TP.periodic}), where
\begin{equation}\label{def.h.rate}
h(x) = \left\{ 
\begin{aligned}
&h_+(x), \qquad \txt{if } x \in {\Omega_+}, \\
&h_-(x), \qquad \txt{if } x \in {\Omega_-},\\
\end{aligned}
\right.
\end{equation}
and $h_\pm(x) \in C^{\alpha}(\overline{\Omega_\pm}; \R^{d\times m})$. Then, as a corollary of Theroem \ref{thm.TP.rate}, we have
\begin{theorem}\label{thm.mainrate}
	Let $u^\e$ and $u^0$ be the weak solutions of (\ref{eq.TP.periodic}) and (\ref{eq.TP.homogenized}) with $F = \nabla\cdot h$, where $h$ takes a form of (\ref{def.h.rate}). 
	
	\begin{itemize}
		\item[(i)] If $\e_{\pm} \in (0,1)$, then there exists $\sigma>0$, depending only on $\Lambda,\alpha$ and the Lipschitz character of $\partial\Omega$, such that
		\begin{equation*}
		\norm {u^\e - u^0}_{L^2(\Omega)} \le C ( \e^\sigma_{+} + \e^\sigma_{-} ) \Big( \sum_{\pm} [h]_{C^{\alpha}(\Omega_{\pm})} + \norm {f}_{H^{1/2}(\partial\Omega)}  + \norm {g}_{H^{-1/2}(S)} \Big).
		\end{equation*}
		
		\item[(ii)] If $\e_+ \in (0,1)$, $\e_- \in (1,\infty)$ and in addition $A_-$ is $C^\alpha$-H\"{o}lder continuous, then for the same $\sigma$ as in part (i),
		\begin{equation*}
		\norm {u^\e - \bar{u}}_{L^2(\Omega)} \le C ( \e^\sigma_{+} + \e^{-\alpha}_{-} ) \Big( \sum_{\pm} [h]_{C^{\alpha}(\Omega_{\pm})} + \norm {f}_{H^{1/2}(\partial\Omega)}  + \norm {g}_{H^{-1/2}(S)} \Big).
		\end{equation*}
	\end{itemize}
\end{theorem}

\begin{proof}
	This may be proved by an interpolation argument. We only prove part (i) as the proof for part (ii) is similar. By subtracting constants, we may assume $h_+(P) = h_-(P) = 0$ for some $P\in S$. Also, we may extend $h_{\pm}$ to the entire $\R^d$ so that the $C^{\alpha}$ character is preserved. Let $t>0$ to be determined and define $h_{\pm,t} = S_t h_\pm$, where $S_t$ is given by (\ref{def.St}), and set
	\begin{equation*}
	h_t(x) = \left\{ 
	\begin{aligned}
	&h_{+,t}(x), \qquad \txt{if } x \in {\Omega_+}, \\
	&h_{-,t}(x), \qquad \txt{if } x \in {\Omega_-}.\\
	\end{aligned}
	\right.
	\end{equation*}
	
	Let $u^\e = u^\e_1 + u^\e_2$, where $u_1^\e$ solves
	\begin{equation}\label{eq.TP.u1}
	\left\{
	\begin{aligned}
	\nabla\cdot (A^\e \nabla u^\e_1) &= \nabla\cdot h_t \qquad &\text{ in }& \Omega \setminus S, \\
	u^\e_1 & = f \qquad &\txt{on } & \partial \Omega, \\
	u^\e_{1,+} &= u^\e_{1,-} \qquad &\text{ on }&  S,\\
	\Big( \frac{\partial u^\e_1}{\partial \nu_\e} \Big)_+ - \Big( \frac{\partial u^\e_1}{\partial \nu_\e} \Big)_- &= g \qquad &\text{ on }& S,
	\end{aligned} 
	\right.
	\end{equation}
	and $u^\e_2$ solves
	\begin{equation}\label{eq.TP.u2}
	\left\{
	\begin{aligned}
	\nabla\cdot (A^\e \nabla u^\e_2) &= \nabla\cdot (h-h_t) \qquad &\text{ in }& \Omega \setminus S, \\
	u^\e_2 & = 0 \qquad &\txt{on } & \partial \Omega, \\
	u^\e_{2,+} &= u^\e_{2,-} \qquad &\text{ on }&  S,\\
	\Big( \frac{\partial u^\e_2}{\partial \nu_\e} \Big)_+ - \Big( \frac{\partial u^\e_2}{\partial \nu_\e} \Big)_- &= 0 \qquad &\text{ on }& S.
	\end{aligned} 
	\right.
	\end{equation}
	Because 
	\begin{equation*}
	\norm {\nabla h_t}_{L^2(\Omega)} \le Ct^{-1} \norm {h}_{L^\infty(\Omega)} \le Ct^{-1} \sum_{\pm} [h]_{C^\alpha(\Omega_{\pm})},
	\end{equation*}
	we obtain from Theorem \ref{thm.TP.rate} that
	\begin{equation*}
	\norm {u^\e_1 - u^0_1}_{L^2(\Omega)} \le C( \e^\sigma_{+} + \e^\sigma_{-} ) \Big( t^{-1} \sum_{\pm} [h]_{C^\alpha(\Omega_{\pm})} + \norm {f}_{H^{1/2}(\partial\Omega)}  + \norm {g}_{H^{-1/2}(S)} \Big),
	\end{equation*}
	where $u^0_1$ is the weak solution of
	\begin{equation}\label{eq.TP.u10}
	\left\{
	\begin{aligned}
	\nabla\cdot (\widehat{A} \nabla u^0_1) &= \nabla\cdot h_t \qquad &\text{ in }& \Omega \setminus S, \\
	u^0_1 & = f \qquad &\txt{on } & \partial \Omega, \\
	u^0_{1,+} &= u^0_{1,-} \qquad &\text{ on }&  S,\\
	\Big( \frac{\partial u^0_1}{\partial \nu_\e} \Big)_+ - \Big( \frac{\partial u^0_1}{\partial \nu_\e} \Big)_- &= g \qquad &\text{ on }& S.
	\end{aligned} 
	\right.
	\end{equation}
	
	Next, we need to estimate $u^\e_2$. To this end, integrating the system (\ref{eq.TP.u2}) against $u^\e_2$ (or see (\ref{def.var.sol})), we have
	\begin{equation*}
	\int_{\Omega} A^\e \nabla u^\e_2 \cdot \nabla u^\e_2 = \int_{\Omega} (h-h_t)\cdot \nabla u^\e_2 - \sum_{\pm} \int_{S} \pm n\cdot (h-h_t)_{\pm} u^\e_2 d\sigma,
	\end{equation*}
	which leads to
	\begin{equation*}
	\norm {\nabla u^\e_2}_{L^2(\Omega)} \le C\norm {h - h_t}_{L^\infty(\Omega)} \le C\sum_{\pm} \norm {h_\pm - h_{\pm,t}}_{L^\infty(\Omega_\pm)}.
	\end{equation*}
	Now, by the $C^\alpha$-H\"{o}lder continuity of $h_\pm$, for every $x\in \Omega_{\pm}$,
	\begin{equation}\label{est.h.int}
	|h_{\pm,t}(x) - h_\pm(x)| = \bigg| t^{-d}\int_{\R^d} (h_\pm(x-y) - h_\pm(x))  \varphi(y/t) dy \bigg| \le Ct^\alpha [h]_{C^\alpha(\Omega_{\pm})}.
	\end{equation}
	Consequently
	\begin{equation}\label{est.ue2}
	\norm {u_2^\e}_{H^1(\Omega)} \le Ct^{\alpha} \sum_{\pm} [h]_{C^{\alpha}(\Omega_{\pm})}.
	\end{equation}
	
	Now, since $u^0$ is the weak solution of (\ref{eq.TP.homogenized}) with $F = \nabla\cdot h$, we may write $u^0 = u^0_1 + u^0_2$ and note that $u_2^0$ solves
	\begin{equation}\label{eq.TP.u20}
	\left\{
	\begin{aligned}
	\nabla\cdot (\widehat{A} \nabla u^0_2) &= \nabla\cdot (h-h_t) \qquad &\text{ in }& \Omega \setminus S, \\
	u^0_2 & = 0 \qquad &\txt{on } & \partial \Omega, \\
	u^0_{2,+} &= u^0_{2,-} \qquad &\text{ on }&  S,\\
	\Big( \frac{\partial u^0_2}{\partial \nu_\e} \Big)_+ - \Big( \frac{\partial u^0_2}{\partial \nu_\e} \Big)_- &= 0 \qquad &\text{ on }& S.
	\end{aligned} 
	\right.
	\end{equation}
	Then, similar to (\ref{est.ue2}), one has
	\begin{equation*}
	\norm {u_2^0}_{H^1(\Omega)} \le Ct^{\alpha} \sum_{\pm} [h]_{C^{\alpha}(\Omega_{\pm})}.
	\end{equation*}
	Hence,
	\begin{equation*}
	\begin{aligned}
	\norm {u^\e - u^0}_{L^2(\Omega)} & \le \norm {u^\e_1 - u^0_1}_{L^(\Omega)} + \norm {u_2^\e}_{L^2(\Omega)} + \norm{u_2^0}_{L^2(\Omega)} \\
	& \le C( \e^\sigma_{+} + \e^\sigma_{-} ) \big( \norm {f}_{H^{1/2}(\partial\Omega)}  + \norm {g}_{H^{-1/2}(S)} \big) \\
	& \qquad + C\big(( \e^\sigma_{+} + \e^\sigma_{-} ) t^{-1} + t^{\alpha} \big) \sum_{\pm} [h]_{C^{\alpha}(\Omega_{\pm})}.
	\end{aligned}
	\end{equation*}
	Finally, since $t>0$ is arbitrary, the desired estimate follows by choosing $t = ( \e^\sigma_{+} + \e^\sigma_{-} )^{1/(1+\alpha)}$.
\end{proof}

\begin{remark}
	The advantage of the method in this subsection lies in generality and the weak assumptions on the interface and transmission value. The pay off is that the convergence rates are far from sharp.
\end{remark}

\section{Uniform Lipschitz estimate}
The uniform Lipschitz estimate in this section is proved by a recently developed quantitative method in homogenization \cite{AS16,Sh17}. We follow the elegant framework formulated by Z. Shen in \cite{Sh17}. Throughout this section, we assume $0<\e_+ \le \e_- < 1$. 

\subsection{Interface stability}
In this subsection, we prove an interface stability result which allows us to ``flatten'' the interface. This step is crucial because the piecewise linear solutions in $\mathscr{L}$ all have the flat interface $\{x_d = 0 \}$.
By translation and rotation, we assume $0\in S$ and $\{x_d = 0\}$ is the tangent plane of $S$ at $0$. Suppose $S$ is given by the graph $x_d = \psi(x')$.
Instead of working on balls, we will work on cylinders for some technical reason. Let $Q_t: = \{ x = (x',x_d): |x'|<t, |x_d| < t  \}$ and $Q_{t,\pm} = \{x\in Q_t: \pm x_d > 0\}$. Let $D_{t,\pm}: = \{ x\in Q_t: \pm x_d > \psi(x') \}$, $S_t = Q_t \cap S$ and $T_t = Q_t \cap \{x_d = 0\}$.

Let $u^0\in H^1(Q_t;\R^m)$ be a weak solution of the following problem
\begin{equation}\label{eq.TP.curved}
\left\{
\begin{aligned}
\nabla\cdot (\widehat{A} \nabla u^0) &= \nabla\cdot h \qquad &\text{ in }& Q_t \setminus S_t, \\
u^0_+ &= u^0_- \qquad &\text{ on }&  S_t,\\
\Big( \frac{\partial u^0}{\partial \nu_0} \Big)_+ - \Big( \frac{\partial u^0}{\partial \nu_0} \Big)_- &= g \qquad &\text{ on }& S_t.
\end{aligned} 
\right.
\end{equation}
The idea is that the above problem may be approximated by the following problem with flat interface
\begin{equation}\label{eq.TP.flat}
\left\{
\begin{aligned}
\nabla\cdot (\widehat{A} \nabla v_t) &= 0 \qquad &\text{ in }& Q_t \setminus T_t, \\
v_t &= u^0 \qquad &\txt{ on } & \partial Q_t, \\
(v_t)_+ &= (v_t)_- \qquad &\text{ on }&  T_t,\\
\Big( \frac{\partial v_t}{\partial \nu_0} \Big)_+ - \Big( \frac{\partial v_t}{\partial \nu_0} \Big)_- &= g_0:=g(0) \qquad &\text{ on }& T_t,
\end{aligned} 
\right.
\end{equation}
Precisely, we prove
\begin{lemma}\label{lem.interface stability}
	Let $u^0$ and $v_t$ solve (\ref{eq.TP.curved}) and (\ref{eq.TP.flat}), respectively. Then
	\begin{equation*}
	\begin{aligned}
	& \frac{1}{t}\bigg( \fint_{Q_t} |u^0 - v_t| \bigg)^{1/2} + \bigg( \fint_{Q_t} |\nabla (u^0 - v_t)| \bigg)^{1/2} \\
	&\qquad \le Ct^{\alpha} \Big( [g]_{C^\alpha(S_t)} + \sum_{\pm} [h]_{C^\alpha(D_{t,\pm})} \Big)  + Ct^{\alpha/2} |g(0)|.
	\end{aligned}
	\end{equation*}
\end{lemma}
\begin{proof}
	Without loss of generality, we assume $h_\pm(0) = 0$. Let $w = u_0 - v_t$. Then $w$ is the weak solution of
	\begin{equation}\label{eq.TP.u0-vt}
	\left\{
	\begin{aligned}
	\nabla\cdot (\widehat{A} \nabla w) &= F \qquad &\text{ in }& Q_t, \\
	w &= 0 \qquad &\txt{ on } & \partial Q_t. \\
	\end{aligned} 
	\right.
	\end{equation}
	where
	\begin{equation*}
	\begin{aligned}
	F & = \big( -g + n\cdot h_+ - n \cdot h_-\big) d\sigma|_{S_t} + g_0 d\sigma|_{T_t} + \nabla\cdot h.
	\end{aligned}
	\end{equation*}
	Testing the system against $w$, we have
	\begin{equation}\label{est.stable.w}
	\begin{aligned}
	\int_{Q_t} \widehat{A} \nabla w\cdot \nabla w & = -\int_{S_t} \big(g_0-g + n\cdot h_+ - n \cdot h_-\big) \cdot w d\sigma  \\
	&\qquad + \bigg[ \int_{S_t} g_0 \cdot w d\sigma -\int_{T_t} g_0 \cdot w d\sigma \bigg]  - \int_{Q_t} h\cdot \nabla w.
	\end{aligned}
	\end{equation}
	
	To estimate the first integral on the right-hand side, note that $h_{\pm}(0) = 0$ and $g_0 = g(0)$. Then, by the $C^{\alpha}$-H\"{o}lder continuity of $g$ and $h$, and the trace theorem, one has
	\begin{equation*}
	\begin{aligned}
	\bigg| \int_{S_t} \big(g_0-g + n\cdot h_+ - n \cdot h_-\big) \cdot w d\sigma \bigg| & \le Ct^\alpha t^{(d-1)/2} \Big( [g]_{C^\alpha(S_t)} + \sum_{\pm} [h]_{C^\alpha(S_t)} \Big) \norm {w}_{L^2(S_t)} \\
	& \le Ct^{d/2+\alpha} \Big( [g]_{C^\alpha(S_t)} + \sum_{\pm} [h]_{C^\alpha(S_t)} \Big) \norm {\nabla w}_{L^2(Q_t)}.
	\end{aligned}
	\end{equation*}
	Also, the last integral of (\ref{est.stable.w}) is easily bounded by the same quantity as above. The slightly tricky one is the estimate of the terms in the brackets. This can be estimated by taking advantage of the $C^{1,\alpha}$ smoothness of the interface $S_t$ and a change of variables. In fact, using
	\begin{equation*}
	d\sigma|_{S_t} = \sqrt{1+|\nabla \psi(x')|^2} d\sigma|_{T_t},
	\end{equation*}
	we may write
	\begin{equation*}
	\int_{S_t} g_0 \cdot w d\sigma = \int_{T_t} g_0 \cdot w(x',\psi(x')) \sqrt{1+|\nabla \psi(x')|^2} d\sigma.
	\end{equation*}
	It follows that
	\begin{equation*}
	\begin{aligned}
	\bigg| \int_{S_t} g_0 \cdot w d\sigma -\int_{T_t} g_0 \cdot w d\sigma \bigg| & \le \int_{T_t} |g_0| |w(x',\psi(x')) - w(x',0)| \sqrt{1+|\nabla \psi(x')|^2} d\sigma \\
	&  \qquad + \int_{T_t} |g_0| |w(x',0)| \Big( \sqrt{1+|\nabla \psi(x')|^2} - 1\Big) d\sigma \\
	& =: I_1 + I_2.
	\end{aligned}
	\end{equation*}
	By the fundamental theorem of calculus,
	\begin{equation*}
	\begin{aligned}
	|w(x',\psi(x')) - w(x',0)| & \le \bigg| \int_0^{\psi(x')}  \frac{\partial}{\partial x_d} w(x',s) ds\bigg| \\
	& \le \int_{-Ct^{1+\alpha}}^{Ct^{1+\alpha}}  \bigg|\frac{\partial}{\partial x_d} w(x',s)\bigg| ds,
	\end{aligned}
	\end{equation*}
	where we have used the fact that $|\psi(x')| \le Ct^{1+\alpha}$. Hence, by the H\"{o}lder's  inequality,
	\begin{equation*}
	\begin{aligned}
	I_1 & \le C |g_0| \int_{T_t} \int_{-Ct^{1+\alpha}}^{Ct^{1+\alpha}}  \bigg|\frac{\partial}{\partial x_d} w(x',s)\bigg| dsd\sigma \\
	& \le C|g_0| t^{(d+\alpha)/2} \norm {\nabla w}_{L^2(Q_t)}.
	\end{aligned}
	\end{equation*}
	
	To estimate $I_2$, note that
	\begin{equation*}
	\big| \sqrt{1+|\nabla \psi(x')|^2} - 1 \big| \le C|\nabla \psi(x')| \le Ct^\alpha.
	\end{equation*}
	Thus, the trace theorem implies that
	\begin{equation*}
	I_2 \le Ct^{d/2+\alpha} |g_0| \norm {\nabla w}_{L^2(Q_t)}.
	\end{equation*}
	Combining the estimates of $I_1$ and $I_2$, we obtain
	\begin{equation*}
	\bigg| \int_{S_t} g_0 \cdot w d\sigma -\int_{T_t} g_0 \cdot w d\sigma \bigg| \le C t^{(d+\alpha)/2} |g_0| \norm {\nabla w}_{L^2(Q_t)}.
	\end{equation*}
	As a result, one sees from (\ref{est.stable.w}) that
	\begin{equation*}
	\begin{aligned}
	\bigg( \fint_{Q_t} |\nabla w|^2 \bigg)^{1/2} & \le Ct^{\alpha} \Big( [g]_{C^\alpha(S_t)} + \sum_{\pm} [h]_{C^\alpha(D_{t,\pm})} \Big) + Ct^{\alpha/2}|g(0)|.
	\end{aligned}
	\end{equation*}
	This implies the desired estimate, due to the Poincar\'{e} inequality.
\end{proof}

\begin{remark}
	Even though the stability lemma was stated and proved for the problem with coefficient tensor $\widehat{A}$, it holds actually for any constant coefficient tensor.
\end{remark}

\subsection{Excess and its properties}
As we have seen in Section 2, certain excess decay estimate is crucial in Schauder theory. To begin with, we state a lemma analogous to Lemma \ref{lem.C1a.excess} for system (\ref{eq.TP.flat}). This is the only place that we use the $C^{1,\alpha}$ estimate of the homogenized problems.

\begin{lemma}\label{lem.vt}
	Let $v_t$ be as above. Then,
	for any $\rho \in (0,t/2)$,
	\begin{equation*}
	\begin{aligned}
	& \inf_{\ell \in \mathscr{L}} \bigg\{ \bigg( \fint_{Q_{\rho}} |v_t - \ell|^2 \bigg)^{1/2} + \rho \norm {g - \mathscr{T}(\ell)}_{L^\infty(S_\rho)} \} \\
	& \qquad \le C(\rho/t)^{1+\alpha} \inf_{\ell \in \mathscr{L}} \bigg\{ \bigg( \fint_{Q_{t}} |v_t - \ell|^2 \bigg)^{1/2} + t \norm {g - \mathscr{T}(\ell)}_{L^\infty(S_t)} \bigg\} + C\rho^{1+\alpha} [g]_{C^\alpha(S_t)}.
	\end{aligned}
	\end{equation*}
	where $\beta \in (0,1)$.
\end{lemma}
\begin{proof}
	First, using Lemma \ref{lem.C1a.excess}, we have
	\begin{equation*}
	\begin{aligned}
	& \inf_{\ell \in \mathscr{L}} \bigg\{ \bigg( \fint_{Q_{\rho}} |v_t - \ell|^2 \bigg)^{1/2} + \rho \norm {g_0 - \mathscr{T}(\ell)}_{L^\infty(S_\rho)} \} \\
	& \qquad \le C(\rho/t)^{1+\alpha} \inf_{\ell \in \mathscr{L}} \bigg\{ \bigg( \fint_{Q_{t}} |v_t - \ell|^2 \bigg)^{1/2} + t \norm {g_0 - \mathscr{T}(\ell)}_{L^\infty(S_t)} \bigg\}.
	\end{aligned}
	\end{equation*}
	Then we only need to replace $g_0$ by $g$ in the last inequality. This is clear if one notices that $\norm {g -g_0}_{L^\infty(S_\rho)} \le C\rho^\alpha [g]_{C^{\alpha}(S_\rho)}$ for any $\rho>0$.
\end{proof}

The above lemma inspires us to define the excess for an arbitrary function $u\in L^2(\Omega;\R^m)$ by
\begin{equation*}
\begin{aligned}
H(u;t) := \inf\limits_{\substack{\ell \in \mathscr{L}}}  \bigg\{ \frac{1}{t}   \bigg( \fint_{Q_t} |u - \ell |^2 \bigg)^{1/2} + \Norm {g - \mathscr{T}(\ell) }_{L^\infty(S_t)} \bigg\},
\end{aligned}
\end{equation*}
where $\mathscr{L}$ and $\mathscr{T}(\ell)$ are defined in Definition \ref{def.Lset} and $\ell(x): = Mx + q = M_+ x \mathbbm{1}_{\{x_d>0\}} + M_- x \mathbbm{1}_{\{x_d<0\}} + q$.
Another quantity involved is
\begin{equation*}
\Phi(u;t) := \inf_{ q \in\R^d} \bigg\{ \frac{1}{t}  \bigg( \fint_{Q_t} |u - q|^2 \bigg)^{1/2} \bigg\} + \Norm {g}_{L^\infty(S_t)}.
\end{equation*}
Observe that, for any $u\in L^2(Q_1;\R^m)$ and $t\in (0,1)$, we have
\begin{equation}\label{est.H<Phi}
H(u;t) \le \Phi(u;t).
\end{equation}
Some properties of $H$ and $\Phi$ are given below.
\begin{proposition}\label{prop.HPhi}
	For any function $u\in L^2(Q_1;\R^m)$, there exists a measurable function $h = h_u:(0,1) \to \R_+$ such that
	\begin{equation}\label{est.H.Phi}
	\left\{
	\begin{aligned}
	\Phi(u;r) &\le H(u;r) + C h(r) \\
	h(r) & \le C(H(u;r) + \Phi(u;r)) \\
	\sup_{r\le t \le 2r} H(u;t) &\le CH (u;2r)\\
	\sup_{r\le t \le 2r} \Phi(u;t) &\le C\Phi (u;2r)\\
	\sup_{r\le s,t \le 2r} | h(s) - h(t)| &\le C H(u;2r), 
	\end{aligned} 
	\right.
	\end{equation}
	for any $t\in (0,1/2)$.
\end{proposition}
\begin{proof}
	Let $\ell_t$ be the piecewise linear solution in $\mathscr{L}$ that minimizes $H(u,t)$. Define
	\begin{equation*}
	h(t) = \bigg( \fint_{Q_t} |\nabla \ell_t|^2 \bigg)^{1/2}.
	\end{equation*}
	Note that if $\ell_t = M_+ x \mathbbm{1}_{\{x_d>0\}} + M_- x \mathbbm{1}_{\{x_d<0\}} + q$, then $\nabla \ell_t = M_+^* \mathbbm{1}_{\{x_d>0\}} + M_-^* \mathbbm{1}_{\{x_d<0\}} $ and $h(t) \simeq |M_+| + |M_-|$.
	Hence, the triangle inequality yields
	\begin{equation*}
	\Phi(u;r) \le H(u;r) + \frac{1}{r}\bigg( \fint_{Q_r} |x\cdot \nabla \ell_r|^2 \bigg)^{1/2} + |\mathscr{T}(\ell_r)| \le H(u;r) + Ch(r).
	\end{equation*}
	This is the first part in (\ref{est.H.Phi}). The second part in (\ref{est.H.Phi}) follows similarly by the triangle inequality.
	
	Now, we show the third part of (\ref{est.H.Phi}). If $r\le t\le 2r$,
	\begin{equation*}
	\begin{aligned}
	H(u;t) & = \frac{1}{t}   \bigg( \fint_{Q_t} |u - \ell_t |^2 \bigg)^{1/2} + \Norm {g - \mathscr{T}(\ell_t) }_{L^\infty(S_t)} \\
	& \le \frac{1}{t}   \bigg( \fint_{Q_t} |u - \ell_{2r} |^2 \bigg)^{1/2} + \Norm {g - \mathscr{T}(\ell_{2r}) }_{L^\infty(S_t)} \\
	& \le \frac{C}{2r}   \bigg( \fint_{Q_{2r}} |u - \ell_{2r} |^2 \bigg)^{1/2} + \Norm {g - \mathscr{T}(\ell_{2r}) }_{L^\infty(S_{2r})} \\
	& \le CH(u;2r),
	\end{aligned}
	\end{equation*}
	where in the second inequality, we enlarged the region from $Q_t$ to $Q_{2r}$, whose volumes are comparable, so that the norms are also enlarged up to a constant. The forth part of (\ref{est.H.Phi}) follows similarly by enlarging the region (or simply by the first three parts of (\ref{est.H.Phi}) and (\ref{est.H<Phi})).
	
	Finally, to show the last part in (\ref{est.H.Phi}), we use the triangle inequality and the third part of (\ref{est.H.Phi}) and derive
	\begin{equation*}
	\begin{aligned}
	|h(s) - h(t)| & \le C\bigg( \fint_{Q_r} |\nabla \ell_t - \nabla \ell_s|^2 \bigg)^{1/2} \\
	& \le C \inf_{ q \in\R^d} \frac{1}{r} \bigg( \fint_{Q_r} |x\cdot \nabla \ell_t - x\cdot \nabla \ell_s - q|^2 \bigg)^{1/2} \\
	& \le C\inf_{ q \in\R^d} \frac{1}{r} \bigg( \fint_{Q_r} |u - x\cdot \nabla \ell_t - q|^2 \bigg)^{1/2} + C\inf_{ q \in\R^d} \frac{1}{r} \bigg( \fint_{Q_r} |u - x\cdot \nabla \ell_s - q|^2 \bigg)^{1/2} \\
	& \le C\inf_{ q \in\R^d} \frac{1}{t} \bigg( \fint_{Q_t} |u - x\cdot \nabla \ell_t - q|^2 \bigg)^{1/2} + C\inf_{ q \in\R^d} \frac{1}{s} \bigg( \fint_{Q_s} |u - x\cdot \nabla \ell_s - q|^2 \bigg)^{1/2} \\
	& \le CH(u;t) + CH(u;s) \\
	& \le CH(u;2r).
	\end{aligned}
	\end{equation*}
	This completes the proof.
\end{proof}

\subsection{Decay estimates}
The decay estimate relies on the scale $r$ compared to $\e_+$ and $\e_-$. The following lemma deals with the case $0<\e_+<\e_- < r < 1$.
\begin{lemma}\label{lem.excess decay}
	Suppose $0<\e_+ < \e_- < 1$. Let $u^\e$ be a weak solution of (\ref{eq.TP.e}). Then there exists $\theta \in (0,1/4)$ so that
	\begin{equation*}
	H(u^\e;\theta r) \le \frac{1}{2} H(u^\e;r) + C (\e_-/r)^\sigma \Phi(u^\e; 2r) + C r^{\alpha/2} B,
	\end{equation*}
	for every $\e_- < r<1/2$, where
	\begin{equation}\label{def.B}
	B = B(h,g) : = \sum_{\pm} [h]_{C^\alpha(D_{1,\pm})} + \norm {g}_{C^\alpha(S_1)} .
	\end{equation}
\end{lemma}
\begin{proof}
	The proof combines the convergence rate (Theorem \ref{thm.mainrate}), interface stability (Lemma \ref{lem.interface stability}) and excess decay estimate for the homogenized equations with flat interface (Lemma \ref{lem.vt}). 
	Let $r < 1/2$ be fixed. Let $u^0$ be the weak solution of
	\begin{equation}\label{eq.TP.u0rate}
	\left\{
	\begin{aligned}
	\nabla\cdot (\widehat{A} \nabla u^0) &= \nabla\cdot h \qquad &\text{ in }& Q_r \setminus S_r, \\
	u^0 &= u^\e \qquad & \txt{ on } & \partial Q_r,\\
	u^0_+ &= u^0_- \qquad &\text{ on }&  S_r,\\
	\Big( \frac{\partial u^0}{\partial \nu_0} \Big)_+ - \Big( \frac{\partial u^0}{\partial \nu_0} \Big)_- &= g \qquad &\text{ on }& S_r.
	\end{aligned} 
	\right.
	\end{equation}
	Since $0<\e_+ \le \e_- < r < 1$, by Theorem \ref{thm.mainrate} part (i) and rescaling, we have
	\begin{equation*}
	\bigg( \fint_{Q_r} |u^\e - u^0|^2 \bigg)^{1/2} \le C r(\e_-/r)^\sigma \bigg\{  r^\alpha \sum_{\pm} [h]_{C^\alpha(Q_{2r,\pm})} + \norm {g}_{L^\infty(S_{2r})} + \frac{1}{r}\bigg( \fint_{Q_{2r}} |u^\e|^2 \bigg)^{1/2} \bigg\}.
	\end{equation*}
	Note that $u^\e - q$ is also a weak solution for any $q\in \R^d$. Then, in view of the definition of $\Phi(u^\e;t)$, we obtain
	\begin{equation}\label{est.ue-u0}
	\frac{1}{r}\bigg( \fint_{Q_r} |u^\e - u^0|^2 \bigg)^{1/2} \le C(\e_-/r)^\sigma \Phi(u^\e;2r) + r^\alpha \sum_{\pm} [h]_{C^\alpha(Q_{1,\pm})}.
	\end{equation}
	
	Next, let $v_r$ be the weak solution of (\ref{eq.TP.flat}). We derive from the interface stability result (Lemma \ref{lem.interface stability}) that
	\begin{equation}\label{est.u0-vr}
	\frac{1}{r}\bigg( \fint_{Q_r} |u^0 - v_r| \bigg)^{1/2} \le Cr^{\alpha/2} \Big(\sum_{\pm} [h]_{C^\alpha(Q_{1,\pm})} + \norm {g}_{C^\alpha(S_1)} \Big).
	\end{equation}
	
	Finally, applying Lemma \ref{lem.vt} and using the definition of $H(v_r;t)$, we have that for any $\theta \in (0,1)$
	\begin{equation*}
	H(v_r;\theta r) \le C\theta^\alpha H(v_r; r) + C r^\alpha [g]_{C^\alpha(S_1)}.
	\end{equation*}
	Choosing and fixing $\theta \in (0,1)$ small enough so that
	$C\theta^\alpha = 1/2$, then
	\begin{equation}\label{est.vr.C1a}
	H(v_r;\theta r) \le \frac{1}{2} H(v_r; r) + C r^\alpha [g]_{C^\alpha(S_1)}.
	\end{equation}
	
	Combining (\ref{est.ue-u0}), (\ref{est.u0-vr}) and (\ref{est.vr.C1a}), we obtain
	\begin{equation*}
	H(u^\e;\theta r) \le \frac{1}{2} H(u^\e; r)  + C(\e_-/r)^\sigma \Phi(u^\e;2r)+ Cr^{\alpha/2} B,
	\end{equation*}
	as desired.
\end{proof}

The next lemma deals with the case $0< \e_+ < r < \e_- < 1$.
\begin{lemma}\label{lem.excess decay2}
	Suppose $0<\e_+ < \e_- < 1$ and $A_-$ is $C^\alpha$-H\"{o}lder continuous. Let $u^\e$ be a weak solution of (\ref{eq.TP.e}). Then there exists $\theta \in (0,1/4)$ so that
	\begin{equation*}
	H(u^\e;\theta r) \le \frac{1}{2} H(u^\e;r) + C \big\{ (\e_-/r)^{-\alpha} + (\e_+/r)^\sigma \big\} \Phi(u^\e; 2r) + C r^{\alpha/2} B,
	\end{equation*}
	for every $\e_+ < r< \e_-$, where $B$ is the same as (\ref{def.B}).
\end{lemma}
\begin{proof}
	The proof is similar to Lemma \ref{lem.excess decay}. Let $\e_+ < r<\e_-$. Let $\bar{u}$ be the weak solution of
	\begin{equation}\label{eq.TP.ubar.rate}
	\left\{
	\begin{aligned}
	\nabla\cdot (\overline{A} \nabla \bar{u}) &= \nabla\cdot h \qquad &\text{ in }& Q_r \setminus S_r, \\
	\bar{u} &= u^\e \qquad & \txt{ on } & \partial Q_r,\\
	\bar{u}_+ &= \bar{u}_- \qquad &\text{ on }&  S_r,\\
	\Big( \frac{\partial \bar{u}}{\partial \nu_0} \Big)_+ - \Big( \frac{\partial \bar{u}}{\partial \nu^\circ} \Big)_- &= g \qquad &\text{ on }& S_r.
	\end{aligned} 
	\right.
	\end{equation}
	Applying Theorem \ref{thm.mainrate} part (ii), we have
	\begin{equation*}
	\begin{aligned}
	\bigg( \fint_{Q_r} |u^\e - \bar{u}|^2 \bigg)^{1/2} & \le C r \big\{ (\e_-/r)^{-\alpha} + (\e_+/r)^\sigma \big\} \\
	& \qquad \times \bigg\{  r^\alpha \sum_{\pm} [h]_{C^\alpha(Q_{2r,\pm})} + \norm {g}_{L^\infty(S_{2r})} + \frac{1}{r}\bigg( \fint_{Q_{2r}} |u^\e|^2 \bigg)^{1/2} \bigg\}.
	\end{aligned}
	\end{equation*}
	Note that $\overline{A}$ is constant. So the solution $\bar{u}$ satisfies the same estimates as $u^0$ in Lemma \ref{lem.excess decay}. Then the desired estimate follows exactly by the same argument.
\end{proof}

\subsection{Iteration and conclusion}

The following iteration lemma is a generalized version of \cite[Lemma 8.5]{Sh17}. The proof is similar to \cite[Lemma 6.7]{GZ20} with obvious modifications.

\begin{lemma}\label{lem.Hiteration}
	Suppose that both $\omega_1$ and $\omega_2$ are non-negative increasing functions satisfying
	\begin{equation*}
	\sum_{i = 1}^2 \int_{0}^{1} \frac{\omega_i(r)}{r} dr < \infty.
	\end{equation*}
	Let $H(r),\Phi(r)$ and $h(r)$ be non-negative quantities such that there exist some constants $\theta\in (0,1/2), \sigma,\beta \in (0,1), C_0 > 0$ and $B>0$ so that for any $r\in (\delta,R)$,
	\begin{equation*}
	\left\{ 
	\begin{aligned}
	H(\theta r) &\le \frac{1}{2} H(r) + C_0 (\omega_1(\delta/r) + \omega_2(r/R)) \Phi(2r) + C_0 r^{\beta} B, \\
	H(r) & \le \Phi(r), \\
	h(r) & \le C_0( H(r) +\Phi(r)), \\
	\Phi(r) &\le H(r) + C_0 h(r), \\
	\sup_{r\le t \le 2r} \Phi(t) &\le C_0 \Phi(2r), \\
	\sup_{r\le s,t \le 2r} | h(s) - h(t)| &\le C_0 H(2r),
	\end{aligned}
	\right.
	\end{equation*}
	where $0<\delta<R\le 1$. 
	Then, there exists $C$ depending only on $\theta,\sigma,\beta,\omega_i$ and $C_0$ such that
	\begin{equation}\label{est.excess decay}
	\int_{\delta}^{R} \frac{H(r)}{r} dr + \sup_{\delta < r<R} \Phi(r)  \le C \Big(\Phi(R) + B \Big).
	\end{equation}
	The point here is that $C$ is independent of $\delta$ and $R$.
\end{lemma}

Now, we are ready to prove the main theorem.
\begin{proof}[Proof of Theorem \ref{thm.Lip}]
	Part (i): 
	It follows from Proposition \ref{prop.HPhi} and Lemma \ref{lem.excess decay} that $H(r) = H(u_\e;r)$, $\Phi(r) = \Phi(u^\e;r)$ and $h(r)$ (defined in Lemma \ref{prop.HPhi}) satisfy the assumptions in Lemma \ref{lem.Hiteration} with
	\begin{equation*}
	\omega_1(r) = r^\sigma,  \quad \delta = \e_- \quad \txt{and} \quad R = \sqrt{2}/2.
	\end{equation*}
	Thus, (\ref{est.H<Phi}) gives
	\begin{equation}\label{est.e-to1}
	\sup_{\e_- < r<\sqrt{2}/2} \Phi(r)  \le C \Big(\Phi(\sqrt{2}/2) + B \Big),
	\end{equation}
	where $B$ is given in (\ref{def.B}). Note that $B_r \subset Q_r \subset B_{\sqrt{2}r}$.
	Then, by the standard Caccioppoli inequality for (\ref{eq.TP.e}) (see Remark \ref{rmk.Caccioppoli}) and the Poincar\'{e} inequality, we obtain (\ref{est.Lip}).
	
	Part (ii): In this case, Lemma \ref{lem.excess decay} and \ref{lem.excess decay2} both hold. Thus, we still have (\ref{est.e-to1}) which gives the estimate down to $\e_-$. To obtain an estimate further down to $\e_+$, we may apply Lemma \ref{lem.Hiteration} once more with
	\begin{equation*}
	\omega_1(r) = r^\sigma, \quad \omega_2(r) = r^\alpha,
	\quad \delta = \e_+ \quad \txt{and} \quad R = \e_-.
	\end{equation*}
	Hence, (\ref{est.H<Phi}) and (\ref{est.e-to1}) combined give
	\begin{equation*}
	\sup_{\e_+ < r<\e_-} \Phi(r)  \le C \Big(\Phi(\e_-) + B \Big) \le C \Big( \Phi(\sqrt{2}/2) + B\Big).
	\end{equation*}
	This implies the desired estimate.
	
	Part (iii): With Part (ii) at our disposal, it suffices to show
	\begin{equation*}
	|\nabla u^\e(0)| \le C \bigg\{ \bigg( \fint_{Q_{\e_+}} |\nabla u^\e|^2 \bigg)^{1/2} + B \bigg\}.
	\end{equation*}
	Indeed, because both $A_+$ and $A_-$ are $C^\alpha$-H\"{o}lder continuous, this follows from Theorem \ref{thm.Cka} with $k = 1$ and a blow-up argument.
\end{proof}

\begin{remark}\label{rmk.Caccioppoli}
	The standard Caccioppoli inequality for (\ref{eq.TP.e}) is as follows:
	\begin{equation*}
	\bigg( \fint_{Q_{r}} |\nabla u^\e|^2 \bigg)^{1/2} \le C \bigg\{ \frac{1}{r} \bigg( \fint_{Q_{2r}} |u^\e|^2 \bigg)^{1/2} + \sum_{\pm} r^\alpha [h]_{C^\alpha(D_{2r,\pm})} + \norm {g}_{L^\infty(S_{2r})} \bigg\}.
	\end{equation*}
	This can be proved by a routine argument combined with the trace theorem.
\end{remark}

\bibliographystyle{amsplain}
\bibliography{mybib}

\end{document}